\documentclass{article}

\usepackage[utf8]{inputenc}

\newcommand{\voorDickOrNotVoorDick}[1]{}

\newcommand{\uu}{{\sf U}}

\newcommand{\idel}[1]{{\ensuremath {\mathrm{I}\Delta_{#1}}}\xspace}

\usepackage{amsthm,bm, amssymb, stmaryrd}
\usepackage{proof}	
\usepackage{color}
\usepackage{amsmath}
\usepackage{gastex}
\usepackage{multicol}
\usepackage{graphicx}
\usepackage{float}

\newcommand{\la}{\langle}
\newcommand{\ra}{\rangle}

\newcommand{\principle}[1]{\ensuremath{\formal{#1}}}

\newcommand{\formal}[1]{\ensuremath{{\sf {#1}}}\xspace}

\newcommand{\il}{{\ensuremath{\textup{\textbf{IL}}}}\xspace}
\newcommand{\extil}[1]{\ensuremath{\textup{\textbf{IL}}{\sf\ensuremath{#1}}}\xspace}
\newcommand{\gl}{{\ensuremath{\textup{\textbf{GL}}}}\xspace}

\newcommand{\intl}[1]{{\ensuremath {\textup{\textbf{IL}}}({\rm #1})}}

\newcommand{\ilm}{\extil{M}}
\newcommand{\ilp}{\extil{P}}
\newcommand{\ilw}{\extil{W}}

\newcommand{\ilal}{\intl{All}\xspace}

\newcommand{\pra}{\ensuremath{{\mathrm{PRA}}}\xspace}
\newcommand{\pa}{\ensuremath{{\mathrm{PA}}}\xspace}

\theoremstyle{plain}
\newtheorem{theorem}{Theorem}[section]
\newtheorem{definition}[theorem]{Definition}
\newtheorem{lemma}[theorem]{Lemma}

\newtheorem{corollary}[theorem]{Corollary}

\theoremstyle{remark}

\theoremstyle{question}

\newcommand{\inty}{interpretability}

\usepackage{xspace}
\usepackage{amssymb}

\usepackage{psfig}
\usepackage{epsfig}

\usepackage{tikz}
\usepackage{mathdots}
\usepackage{cancel}
\usepackage{color}
\usepackage{siunitx}
\usepackage{array}
\usepackage{multirow}
\usepackage{tabularx}
\usepackage{booktabs}
\usepackage{stmaryrd}
\usepackage{wasysym}
\usepackage{mathtools}
\usepackage{amsfonts}
\usetikzlibrary{fadings}

\newcommand{\ilgen}[1]{\text{\textbf{IL}\textsubscript{gen}\textsf{#1}}}
\newcommand{\kgen}[1]{\text{(\textsf{{#1}})\textsubscript{gen}}}

\title{An overview of Generalised Veltman Semantics}
\author{Joost J. Joosten, Jan Mas Rovira, Luka Mikec, Mladen Vuković}

\begin{document}
\maketitle

\begin{abstract}
\voorDickOrNotVoorDick{This paper is dedicated to Dick de Jongh at the occasion of his 80th birthday and gives an 
overview of what is called Generalised Veltman semantics.} Interpretability logics are 
endowed with relational semantics \`a la Kripke: Veltman semantics. For certain 
applications though, this semantics is not fine-grained enough. Back in 1992\voorDickOrNotVoorDick{, in the research group of de Jongh, } the notion of \emph{generalised Veltman semantics} emerged to obtain certain 
non-derivability results as was first presented by Verbrugge (\cite{Verbrugge}). It has turned out that this semantics has various good properties.
In particular, in many cases completeness proofs become simpler and the richer semantics will allow for filtration arguments as opposed to regular 
Veltman semantics. This paper aims to give an overview of results and applications of Generalised 
Veltman semantics up to the current date.
\end{abstract}

\tableofcontents

\section{Introduction}
This paper deals with interpretability logics and a particular kind of relational semantics for them. In the literature we find interpretability logics that are propositional modal logics with a unary modality $\Box$ for formal provability and a binary modality $\rhd$ used to denote formal interpretability or some related (meta)mathematical notion like, for example, conservativity. As such, interpretability logics extend the well-known provability logic \gl. In particular, Kripke semantics for \gl can be extended to account for the new binary interpretability operator $\rhd$. Such an extension was first considered by Frank Veltman and now goes by the name of Veltman semantics. 

Unary modal logics allow for other abstract semantics like topological or neighbourhood semantics which can be seen as generalisations of the regular Kripke semantics. Although these kind of semantics have not yet been studied for interpretability logics, there is some sort of hybrid generalisation which is called \emph{Generalised Veltman Semantics}, GVS for short. GVS \voorDickOrNotVoorDick{originated in the research group of Dick de Jongh and} was first introduced and studied by Verbrugge in \cite{Verbrugge}.
In GVS the $\Box$ modality is dealt with as before and only the $\rhd$ semantics is generalised in a way reminiscent to neighbourhood semantics. 

GVS is more fine-grained than regular Veltman semantics and as such can serve the purpose of distinguishing logical axiom schemes. However, the most important contribution of GVS to the field of interpretability logics lies in the fact that it allows for filtration in various systems, contrary to regular Veltman semantics. This has resulted in a situation where certain logics are known to be complete w.r.t.~GVS but known to be incomplete with respect to regular Veltman semantics. In other cases, the only known way of showing decidability of a logic runs via GVS. Thus, GVS is proving itself an important tool in the study of interpretability logics. Moreover, completeness proof for GVS often turn out to be more uniform and simple than in the case of regular Veltman semantics as one can appreciate in Sections \ref{section:CompletenessPreliminaries} and \ref{section:ModalCompleteness} of this paper.

In the remainder of the paper we will give an overview of the uses and occurrences of GVS. Before doing so, we finish this introduction by briefly outlining the general development of \inty~logics and the role of GVS therein. The main point of the remainder of this introduction is to show that the field of interpretability logics has grown to a mature state having various applications and relations to other fields of logic and we do not claim to give an exhaustive overview. \voorDickOrNotVoorDick{In particular, we hope that Dick de Jongh feels very satisfied of taking notice of the state-of-the-art: to see one of his many intellectual children in early maturity. Thank you for your many contributions Dick, to this beautiful field!}\\
\ \medskip

\subsection{The beginnings}
An interpretation of a theory $V$ into a theory $U$ is roughly a translation $j$ that maps the non-logical symbols in the language of $V$ to formulas in the language of $U$ with exactly the same free variables, so that any theorem $\varphi$ of $V$ will become provable in $U$ too when we apply the translation $j$ to it. Various kinds of interpretations are around in the literature and interpretations are used in practically all branches of mathematics or meta-mathematics though a first methodological treatment is presented in \cite{TarskiEtAl:1953:UndecidableTheories}. We are not too much interested in the variations and uses of interpretations and refer the interested reader to \cite{Visser:1997:OverviewIL, JaparidzeJongh:1998:HandbookPaper}. Rather, we shall focus on the logics that describe the structural behaviour of the interpretability notion.

Interpretability logics arose in the eighties/nineties of the last century around Petr H\'ajek, Dick de Jongh, Franco Montagna, Vít\v{e}zlav \v{S}vejdar, Frank Veltman and Albert Visser. Just like provability logic describes the provably structural behaviour of the formalised provability predicate, the aim was to find a logic that describes exactly all provably structural behaviour of the notion of interpretability. 

Probably the first published conception of interpretability as a modal operator is \v{S}vejdar's \cite{svejdar:1983} from\footnote{We refer in this introduction to the year of publication while often preprints already circulated prior to that.} 1983. 
Montagna took this project further in his paper \cite{Montagna:1987:Provability} of 1987. However, it was not until 1990 when
 Visser conceived  the modal logical framework proposing various principles and axiom schemes (\cite{Visser:1988:preliminaryNotesOnInterpretabilityLogic, Visser:1990:InterpretabilityLogic}) in the format we know today giving rise to the basic interpretability logic \il. Semantics for these systems was provided by Veltman and de Jongh who proved completeness (\cite{JonghVeltman:1990:ProvabilityLogicsForRelativeInterpretability}) in 1990 for the logics \il, and the extensions \ilp and \ilm.
 
 Ever since, the field of interpretability logics knows various different logics and has seen a development into quite a mature field interacting with various other branches of logic and mathematics. We shall now mention a few of the most notable developments. Below we shall mention various systems of interpretability logic, the definitions of which shall be postponed to later sections.
 
 \subsection{Relation to meta-mathematics}
 Various logics akin to \il are known to adequately describe well-delimited parts of meta-mathematics lending much importance and applicability to these logics. The first completeness result
 arose when Visser proved (\cite{Visser:1990:InterpretabilityLogic}, \cite{Zambella:1992:Interpretability}) arithmetical completeness for the logic \ilp for any $\Sigma_1$ sound, finitely axiomatisable theory where the superexponential function is provably total.

Any $\Sigma_1$-sound theory containing a minimum of arithmetic ($\idel{0} +\exp{}$) will have the same provability logic \gl (\cite{de1991proof}). The situation for interpretability is very different. Shavrukov (\cite{Shavrukov:1988:InterpretabilityLogicPA}) and independently Berarducci (\cite{Berarducci:1990:InterpretabilityLogicPA}), established that the logic \ilm generates precisely the set of interpretability principles that are always provable\footnote{The way to go from always provable to always true is done as always by adding reflection over the set of theorems. See \cite{DBLP:journals/jsyml/Strannegard99} for a slight generalisation.} in theories like \pa (being $\Sigma_1$ sound and having full induction/proving consistency of any of its finite subsystems).

Since different kinds of theories have different kind of interpretability logics, a major question in the field revolves about the logic that generates the collection of principles that is always provable in \emph{any} reasonable arithmetical theory. Here the quantifier ``any'' is left vague on purpose but can be taken to be any theory containing $\idel{0} +\exp{}$. The target logic is denoted \ilal and even though much progress has been achieved \cite{JoostenVisser:2000:IntLogicAll,rijk:unar92,GorisJoosten:2011:ANewPrinciple,Joosten:1998:MasterThesis,Joosten:2004:InterpretabilityFormalized, GorisJoosten:2020:TwoSeries} 
on the question, its exact nature is still a major open problem. Given the plethora of recent principles it is quite conceivable that a natural answer to this question will require an extension of the language as in \cite{JoostenVisser:2004:Toolkit, JoostenMikecVisser:2020:TwoLogics}.

Basically, for any theory that is neither essentially reflexive nor finitely axiomatisable and proving the totality of the superexponentiation function, the corresponding interpretability logic is unknown. Some very modest progress is presented on the interpretability logic of \pra (\cite{BilkovaJonghJoosten:2009:PRA,Joosten:2005:ClosedFragmentILPRAwithIsig1,Joosten:2010:ConsistencyInPRAandISIGMA, JoostenIcard:2012:RestrictedSubstitutions}). Kalsbeek has published some notes (\cite{Kalsbeek:1991:TowardsExp}) on the interpretability logic of $\idel{0} +\exp{}$ and probably that is about how much is known about interpretability logics for theories that fall out of the two kinds mentioned above.  

Interpretability logics are also good for describing (meta-)mathematical phenomena other than or variations of interpretability. Without giving the definitions, we mention that propositional modal logics similar to the logics mentioned above with the modalities $\{ \Box, \rhd\}$,  occur when axiomatising phenomena like partial conservativity (\cite{orey:rela61, haje:inte71, haje:inte72, haje:cons90, haje:cons92, Japaridze:1994:SimpleProofPi1Conservativity, Joosten:2016:OreyHajek, igna:part91}), cointerpretability, tolerance and cotolerance (\cite{Japaridze:1992:ToleranceLogic, Japaridze:1993:WeakInterpretability, JaparidzeJongh:1998:HandbookPaper}), $\Sigma_1$-interpolability (\cite{igna:prov93}), constructive preservativity (\cite{iemh:prop05}), intuitionistic interpretability (\cite{LitakV18:im, Litak14:trends, litak2017constructive}) and feasible interpretability (\cite{verb:feas93}).

\subsection{Abstract semantics} Even though interpretability has been studied from a perspective of categories (\cite{viss:cate06}) and degrees (\cite{ Svejdar:1978:DegreesOfInterpretability, Mycielski:latticeOfChapters:1990, Bennet:1986Orderings, lind:aspe97, viss:faith05, Visser:2014:IntyDegreesFiniteSequential}) in quite some depth, no abstract algebraic, nor topological semantics for interpretability logics have been studied or designed yet. However, apart from the above-mentioned arithmetical semantics, interpretability logics and its kin come with natural relational semantics. 
This relational so-called \emph{Veltman semantics} has seen a considerable development over the past decades. 

Completeness proofs and decidability results for various logics abounded since de Jongh and Veltman's first results (\cite{JonghVeltman:1990:ProvabilityLogicsForRelativeInterpretability}) using a kind of semi-canonical model definition to deal with the logics \il, \ilm and \ilp. Further completeness results yielded ever more complicated proofs culminating in a completeness proof of \ilw in \cite{JonghVeltman:1999:ILW}. 

New model-theoretical techniques were needed to address other systems like the \emph{step-by-step method}\footnote{These step-by-step methods were around and entering main-stream modal-logical literature since the seventies. They were coined as such in \cite{Burgess02} but before that also went by different names as \emph{Completeness by construction} in a reader from de Jongh and Veltman from the eighties \cite{JonghVeltman83IntensionalLogicCourseText}. The first time the step-by-step method was applied to interpretability logics however, are the completeness proofs for \ilw by de Jongh and Veltman (\cite{JonghVeltman:1999:ILW}) and completeness proofs for \extil{M} and \extil{M_0} by Joosten (\cite{Joosten:1998:MasterThesis}) with later a corrected proof for \extil{M_0} by Goris and Joosten in \cite{Goris-Joosten-08}. Joosten wrote his master thesis under direction of de Jongh and  Visser. De Jongh suggested a variation of the step-by-step method so that we would work with infinite maximal consistent sets but only using finitely many of those to build models. In the context of this historic digression it is good to mention Verbrugge's master thesis \cite{Verbrugge1988MastersThesis} under direction of de Jongh and Visser where she proves completeness w.r.t.~so-called umbelliferous frames using a step-by-step method via the intersection of infinite maximal consistent sets with a finite fragment for a logic $\sf Umb$ which contains modalities for provability, interpretability and witness comparisons as in \v{S}vejdar's \cite{svejdar:1983}.\label{footnoteOnSteps}} (\cite{JonghVeltman:1999:ILW, Goris-Joosten-08, Joosten:1998:MasterThesis}) for logics as \extil{M_0} and \extil{W^*} and \emph{assuring labels} for substantial simplifications of general proofs but most notably \ilw (\cite{BilkvaGorisJoosten:2004:SmartLabels, goris2020assuring}). 

Various other logics resisted completeness proofs for Veltman semantics and actually many turned out to be incomplete. Here the technique of generalised Veltman semantics came to the rescue which is the main topic of this paper. Recently, also subsystems of \il are  being studied where the relational semantics again turn out to be adequate \cite{kurahashi2020modal}.

The notion of bisimulation could be applied to both Veltman semantics \cite{Visser:1990:InterpretabilityLogic} and GVS \cite{Vrgoc-Vukovic}. Furthermore, in \cite{CacicVrgoc:2013}, \v{C}a\v{c}i\'{c} and Vrgo\v{c} defined the notion of a game for Veltman models and proved that a winning strategy for the defender in such a game is equivalent to picking out a bisimulation between two models.  

In various occasions, results using modal model techniques could be translated back to arithmetic results like in the case of self provers and $\Sigma_1$ sentences (\cite{dejo:solu98, GorisJoosten:2012:SelfProvers}) and the fixpoint theorem \cite{JonghVeltman:1990:ProvabilityLogicsForRelativeInterpretability}
and
\cite{deJongh-Visser-91}.
 
\subsection{Decidability and complexity} For many interpretability logics, decidability follows directly from 
their modal completeness proofs and the finite model property (FMP). This is the case for \extil{}, \extil{W}, \extil{P}, 
and \extil{M}. 
For other logics, completeness was obtained with the help of constructions that avoid finite 
(truncated) maximal consistent sets \cite{Goris-Joosten-08}, \cite{Mikec-Vukovic-20}. 
In such cases it is not obvious how 
to make the constructions finite.
Here, the method of filtration has proven useful. However, up to today, we only know how to perform filtration for GVS and not for regular Veltman semantics. The decidability of the logics \extil{M_0} \cite{Perkov-Vukovic-16}, \extil{W^*} 
\cite{Mikec-Perkov-Vukovic-17}, \extil{P_0} and \extil{R} \cite{Mikec-Vukovic-20} has been 
shown using these methods.
In all these cases finite models are obtained as quotient models w.r.t.\ the largest 
bisimulation of the given generalised Veltman model.
At the time of writing, there is no known example of a complete 
and undecidable logic.

For a long time nothing non-trivial was known regarding complexity of interpretability 
logics. The first result concerning complexity was 
that already the closed fragment of \extil{}, unlike the closed fragment of \gl{}, is 
\textsf{PSPACE}-hard \cite{BouJoosten:2011}.
The only other published result is that \il{} is \textsf{PSPACE}-complete 
\cite{mikec-pakhomov-vukovic}. The third author believes to have shown that \ilp{} and \ilw{} 
are \textsf{PSPACE}-complete too, and hopes to do the same with \ilm{} (the proof for which 
turns out to be significantly harder to complete than expected).

\subsection{Many classical results carry over to interpretability logics} A clear signal of working with the right notion and framework is seen in the fact that various classical results find their analogs in our logics. We already mentioned the Fixpoint Theorem and will in this paragraph mention some others without pretending to give an exhaustive overview. 

Areces, Hoogland, and de Jongh in \cite{ArecesHooglanddeJongh:2001} proved 
that arrow interpolation holds for \il, i.e.\ if  $\vdash_{\il} A\rightarrow B$ then there 
is a formula $I$ in the common language of $A$ and $B$ such that 
$\vdash_{\il} A\rightarrow I$ and $\vdash_{\il} I\rightarrow B.$ 
As corollaries one obtains turnstile interpolation (i.e.\ if $A\vdash_{\il}B$ 
then there is a formula $I$ in common language such that $A\vdash_{\il} I$ and
$I\vdash_{\il} B$) and $\rhd$-interpolation (i.e.\ if $\vdash_{\il} A\rhd B$ then there is a 
formula $I$ in common language such that $\vdash_{\il} A\rhd I$ and
$\vdash_{\il} I\rhd B$ ).
In  \cite{ArecesHooglanddeJongh:2001} it is also shown that all these properties transfer to the
system \il{P}.  

It is proven that the system \il{W} doesn't have the property of arrow 
interpolation. Visser \cite{Visser:1997:OverviewIL} proved that systems between \il{M}$_0$ and 
\il{M} do not have interpolation either, although this can be restored by enriching the language (see \cite{DBLP:journals/ndjfl/Goris06}). The interpolation property for the system \il{F} is an open 
problem. For all provability and interpretability logics it is shown in 
\cite{ArecesHooglanddeJongh:2001}
that the Beth definability property and fixed  points property are interderivable. This implies that 
all extensions of the basic system of provability logic \gl{} and all extensions of \il have 
the Beth property.

Perkov and Vukovi\'c \cite{PerkovVukovic:2014} proved a version of van Benthem's characterisation theorem (see \cite{Benthem1983}) for 
interpretability logic. A first-order formula is equivalent to the standard first-order 
translation of some formula of interpretability logic with respect to Veltman models if and 
only if it is invariant under bisimulations between Veltman models. To prove this, they used 
bisimulation games on Veltman models. They provide characteristic formulas which formalise the 
existence of winning strategies for the defender in finite bisimulation games.

H\'{a}jek and \v{S}vejdar \cite{HajekSvejdar:1991:ClosedFormulasInterpretability} determined normal forms for the closed fragment of the
system \il{F}, and showed that we can eliminate the modal operator $\rhd$ from closed \il-formulas.
The normal form for the closed fragment of \il is unlikely to have a nice solution given the PSPACE completeness. However,
\v{C}a\v{c}i\'{c} and Vukovi\'{c} \cite{CacicVukovic:2012} proved normal forms exist for a wide class of closed \il formulas. \v{C}a\v{c}i\'{c} and Kova\v{c}  \cite{CacicKovac:2014} quantified asymptotically,
in exact numbers, how wide those classes are using results from combinatorics and asymptotic analysis.

\subsection{Proof theory} To the best of our knowledge, very few well-behaved proof systems for interpretability logics have been studied. Sasaki \cite{Sasaki:2002:CutFreeIL} gave a cut-free sequent system for \il and prove a cut-elimination theorem for it.
Hakoniemi and Joosten \cite{HakoniemiJoosten:2016:TableauxForInterpretabilityLogics} give a treatment of labelled tableaux proof systems and uniformly prove completeness for any logic whose frame condition is given by a Horn formula.

\section{Logics for interpretability}

In this section we shall lay down the basic definitions.

\subsection{Modal interpretability logics}

In interpretability logics, we adopt a reading convention due to Dick de Jongh that will allow 
us to omit many brackets. As such, we say that the strongest binding `connectives' are $\neg$, 
$\Box$ and $\Diamond$ which all bind equally strong. Next come $\wedge$ and $\vee$, 
followed by $\rhd$ and the weakest connective is $\to$. Thus, for example, 
$A\rhd B \to A \wedge \Box C \rhd B\wedge \Box C$ will be short for 
$(A\rhd B) \to \big((A \wedge \Box C) \rhd (B\wedge \Box C)\big)$.


We first define the core logic \il which shall be present in any other interpretability 
logic. As before, we work in a propositional signature where apart from the classical 
connectives we have a unary modal operator $\Box$ and a binary modal operator $\rhd$.

\begin{definition}
The logic \il contains apart from all propositional logical tautologies, all instantiations 
of the following axiom schemes:

\begin{enumerate}
\item[${\sf L1}$]\label{ilax:l1}
        \, $\Box(A\rightarrow B)\rightarrow(\Box A\rightarrow\Box B)$
\item[${\sf L2}$]\label{ilax:l2}
        \, $\Box A\rightarrow \Box\Box A$
\item[${\sf L3}$]\label{ilax:l3}
        \, $\Box(\Box A\rightarrow A)\rightarrow\Box A$
\item[${\sf J1}$]\label{ilax:j1}
        \, $\Box(A\rightarrow B)\rightarrow A\rhd B$
\item[${\sf J2}$]\label{ilax:j2}
        \, $(A\rhd B)\wedge (B\rhd C)\rightarrow A\rhd C$
\item[${\sf J3}$]\label{ilax:j3}
        \, $(A\rhd C)\wedge (B\rhd C)\rightarrow A\vee B\rhd C$
\item[${\sf J4}$]\label{ilax:j4}
        \, $A\rhd B\rightarrow(\Diamond A\rightarrow \Diamond B)$
\item[${\sf J5}$]\label{ilax:j5}
        \, $\Diamond A\rhd A$
\end{enumerate}

The rules of the logic are Modus Ponens (from $\vdash A\to B$ and $\vdash A$, conclude $\vdash B$) and Necessitation 
(from $\vdash A$ conclude $\vdash \Box A$).
\end{definition}
Since $\Box$ denotes provability and $\rhd$ denotes interpretability --a base theory together with the left-hand formula interprets the same base theory together with the right-hand formula-- we can already see what the ${\sf J}$-principles above express. Thus, Principle ${\sf J1}$ expresses that the identity translation defines an interpretation. Principle ${\sf J2}$ expresses that one can compose interpretations by applying one translation after another. Principle ${\sf J3}$ is sometimes referred to as H\'ajek's principle reflecting so-called H\'ajek's Theorem\footnote{We thank Vitek \v{S}vejdar for pointing this out to us.} about constructing an interpretation by cases. Principle ${\sf J4}$ reflects that an interpretation gives rise to relative consistency. Finally, ${\sf J5}$ reflects that one can perform the Henkin construction in arithmetic so that consistency provides an inner model from which an interpretation can be distilled.

\subsection{Arithmetical semantics}
Interpretability logics are related to arithmetics in very much the same way as provability logics are. Thus, we define an arithmetical realisation $*$ as a map that takes propositional variables to sentences in the language of arithmetic\footnote{We assume that all our theories contain the language of arithmetic in one way or another. We refer the reader to \cite{JaparidzeJongh:1998:HandbookPaper} for details of this and our Definition \ref{definition:Intl}}. The realisation is extended to act on arbitrary modal formulas by preserving the logical structure thus commuting with the Boolean connectives. The modality $\Box$ is mapped to an arithmetisation of ``is provable in the base theory $T$'' for some fixed base theory $T$ and likewise will $\rhd$ be mapped to an arithmetisation of interpretability. The interpretability logic of a theory $T$ is now defined as usual being the set of modal formulas whose realisation is provable in $T$ regardless on the exact nature of the realisation:

\begin{definition}\label{definition:Intl}
Let $T$ be a theory in the language of arithmetic that is strong enough to allow for a proper treatment of formalised syntax. We define the \emph{interpretability logic} of $T$ as
\[
\intl{T} \ := \ \{ \varphi \mid \forall * T \vdash \varphi^* \}.
\]
\end{definition}

Even though the notion of formalised interpretability is $\Sigma^0_3$-complete (\cite{Shavrukov:1997:ReflexiveInfinitelyManyAxioms}), for two classes of theories we have an elegant and decidable characterisation for the corresponding interpretability logic. We call a theory $T$ \emph{$\Sigma_1$-sound} if it only proves true $\Sigma_1$-sentences.

By \ilm we denote the logic that arises by adding Montagna's axiom scheme 
\[
\principle M \ := \ \ \ A \rhd B \rightarrow A \wedge \Box C \rhd B \wedge \Box C
\]
to \il. 
\begin{theorem}[Berarducci \cite{Berarducci:1990:InterpretabilityLogicPA}, 
Shavrukov\footnote{Shavrukov's arithmetical completeness proof is dated 1988 which is before the 
official publication of the modal completeness proof of de Jongh and Veltman 
\cite{JonghVeltman:1990:ProvabilityLogicsForRelativeInterpretability} and Visser's simplification 
thereof using a single accessibility relation $S$ to model the $\rhd$ modality 
\cite{Visser:1988:preliminaryNotesOnInterpretabilityLogic}. However, these results were already 
available in preprint form as \emph{D.H.J.~de~Jongh, F.J.M.M.~Veltman. Provability logics for relative 
interpretability. ITLI Prepublication Series ML-88-03 (1988)}; and \emph{A.~Visser. Preliminary 
notes on interpretability logic. Logic Group Preprint Series No.29 (1988);} respectively.} 
\cite{Shavrukov:1988:InterpretabilityLogicPA}]\label{theo:shav}
If $T$ is $\Sigma_1$-sound and proves full induction, then $\intl{T}=\ilm$.
\end{theorem}

The logic \ilp arises by adding the axiom scheme
\[
\principle P \ := \ \ \ A \rhd B \rightarrow \Box (A \rhd B)
\]
to the basic logic \il. The logic \ilp is related to finitely axiomatised theories that can prove the totality of ${\tt supexp}$, where ${\tt supexp}(x)$ is defined as $x\mapsto 2^x_x$ 
with $2^n_0:=n$ and $2^{n}_{m+1}:= 2^{(2^n_m)}$.

\begin{theorem}[Visser \cite{Visser:1990:InterpretabilityLogic}]
If $T$ is $\Sigma_1$-sound, finitely axiomatised and proves the totality of ${\tt supexp}$, then $\intl{T}=\ilp$.
\end{theorem}

Since $\ilm \neq \ilp$ it is very natural to ask for the core logic that is contained in the interpretability logic of any (strong enough) theory. This results in an additional quantifier in the definition of what is often called \emph{the interpretability logic of all reasonable arithmetical theories}:

\begin{definition}
\[
\ilal \ := \ \{ \varphi \mid \forall \, T\supseteq \idel{0} + \exp \ \forall * T \vdash \varphi^* \}.
\]
\end{definition}

Since it is well known and easy to see that all theorems of \il hold in any strong enough arithmetical theory, by the above two theorems we obtain that $\il \subseteq \ilal \subseteq (\ilp \cap \ilm)$. As to date, a modal characterisation of \ilal is unknown. Most principles in this paper have been considered because of their relation to \ilal.

\subsection{Relational semantics}

We can equip interpretability logics with a natural relational semantics often referred to 
as Veltman semantics.

\begin{definition}\label{definition:VeltmanFrameAndModel}
A \emph{Veltman frame} is a triple $\la W, R, \{S_w : w\in W\}\ra$ where $W$ is a non-empty 
set of \emph{possible worlds}, $R$ a binary relation on $W$ so that $R^{-1}$ is transitive 
and well-founded. The $\{S_w : w\in W\}$ is a collection of binary relations on $R[w]$ 
(where $R[w]:= \{ v\mid wRv \}$).
The requirements are that the $S_w$ are reflexive and transitive and the restriction of $R$ 
to $R[w]$ is contained in $S_w$, that is $R\cap R[w] \subseteq S_w$. 

A \emph{Veltman model} consists of a Veltman frame together with a valuation 
$V : {\tt Prop} \to \mathcal P (W)$ that assigns to each propositional variable 
$p\in {\tt Prop}$ a set of worlds $V(p)$ in $W$ where $p$ is stipulated to be true. This 
valuation defines a forcing relation $\Vdash \ \subseteq W{\times} {\sf Form}$ telling us 
which formulas are true at which particular world:
\[
\begin{array}{rll}
w\Vdash p & :\Leftrightarrow & w\in V(p);\\
  w\Vdash \bot & &\mbox{ for no $w\in W$};\\
w\Vdash A\to B & :\Leftrightarrow & w\nVdash  A \mbox{ or } w\Vdash B;\\
w\Vdash \Box A & :\Leftrightarrow & \forall v\ ( wRv \Rightarrow v\Vdash  A);\\
w\Vdash A\rhd B & :\Leftrightarrow & \forall u\ \Big( wRu \ \& \ u\Vdash A \Rightarrow 
\exists v \ (uS_w v \ \& \ v\Vdash  B)\Big).
\end{array}
\]
For a Veltman model $\mathfrak{M} = \la W, R, \{S_w : w\in W\}, V\ra$, we shall write 
$\mathfrak{M} \models A$ as short for $(\forall \, w \in W)\ \mathfrak{M}, w \Vdash A$.
\end{definition}

De Jongh and Veltman proved that the logic \il is sound and complete with respect to all Veltman models 
(\cite{JonghVeltman:1990:ProvabilityLogicsForRelativeInterpretability}).

Often one is 
interested in considering all models that can be defined over a frame. Thus, given a frame 
$\mathfrak F$ and a valuation $V$ on $\mathfrak F$ we shall denote the corresponding model by 
$\la \mathfrak{F}, V \ra$. A \emph{frame condition} for an axiom scheme $A$ is a formula 
$(A)$ (first or higher-order) in the language $\{ R, \{S_w : w\in W\} \}$ so that $\mathfrak{F} 
\models (A)$ (as a relational structure) if and only if $\forall^{\sf valuation} V\ \la 
\mathfrak{F},V\ra \models A$.

\section{Generalised Veltman Semantics}


For certain purposes, Veltman semantics is not fine-grained enough. Generalised semantics was originally introduced by Verbrugge \cite{Verbrugge} in 1992 to determine independence 
between certain interpretability logics as we shall discuss in the next section.

\subsection{Replacing worlds by sets of worlds}
The idea of generalised Veltman semantics is that we will use sets to model the $\rhd$ modality. To be more precise, instead of having the $S_w$ be a relation between worlds, we will use a relation between worlds and \emph{sets of worlds}. Thus, we would have things like $uS_wV$ where $V$ is a set of worlds. The forcing relation would be defined accordingly:
\[
w\Vdash A\rhd B \ :\Leftrightarrow \ \forall u \Big( wRu \ \& \ u \Vdash A \Rightarrow \exists V (uS_wV \ \& \ V\Vdash B)  \Big),
\]
where $V\Vdash B$ is short for $(\forall \, v{\in}V\,) v\Vdash B$. 

In doing so, the axiom scheme $A\rhd B \to (\Diamond A \to \Diamond B)$ imposes\footnote{See \cite{kurahashi2020modal} for a more detailed discussion.} that all of $V$ should be $R$-above $w$. The axiom scheme $\Box (A\to B) \to A\rhd B$ requires that the $S_w$ relation is what we call \emph{semi-reflexive} in the sense that $uS_w\{ u \}$ whenever $wRu$.  The axiom scheme $\Diamond A \rhd A$ imposes that whenever $wRuRv$, then $uS_w\{ v \}$. Just like in regular Veltman semantics, the axiom scheme $(A\rhd C) \wedge (B\rhd C) \to A\vee B \rhd C$ does not impose any requirement on generalised Veltman semantics and is satisfied automatically. It turns out that there is quite some freedom in how to account for the axiom scheme $(A\rhd B) \wedge (B\rhd C) \to A \rhd C$. One such choice is the predominant one in the literature and we shall give it here and fix it for the remainder of this paper. The definition was already given in the original document \cite{Verbrugge} by Verbrugge. Variations will be discussed in the next subsection.

\begin{definition}
A \emph{generalised Veltman frame} $\mathfrak F$ is a structure $\la W,R,\{S_w:w\in W\} \ra$, 
where $W$ is a non-empty set, $R$ is a transitive and converse well-founded binary relation on 
$W$ and for all $w\in W$ we have:
	\begin{itemize}
		\item[a)] $S_w\subseteq R[w]\times \left(\mathcal{P}( 
		{R[w]})\setminus\{\emptyset\}\right)$;
		\item[b)] $S_w$ is quasi-reflexive: $wRu$ implies $uS_w\{u\}$;
		\item[c)] $S_w$ is quasi-transitive: if $uS_wV$ and $vS_wZ_v$ for all $v\in V$, then\\
		 $uS_w(\bigcup_{v\in V}Z_v)$;
		\item[d)] if $wRuRv$, then $uS_w\{v\}$;
		\item[e)] monotonicity: if $uS_wV$ and $V\subseteq Z\subseteq R[w]$, then $uS_wZ$.
	\end{itemize}
A \emph{generalised Veltman model} is a quadruple 
$\mathfrak{M}=\la W,R, \{S_w : w\in W \}, V\ra$, where the first three components form a 
generalised Veltman frame and where $V$ is a valuation mapping propositional variables to subsets 
of $W$. 
The forcing relation $\mathfrak{M}, w \Vdash A$ is defined as before in Definition 
\ref{definition:VeltmanFrameAndModel} with the sole difference that now
\[
w\Vdash A\rhd B \ : \Longleftrightarrow \ \forall u\ \Big(wRu \ \& \ u\Vdash A 
\Rightarrow \exists V (uS_w V \ \& \ V\Vdash B)\Big).
\]
\end{definition}

It is easy to see that GVS is adequate for \il:

\begin{theorem}\label{theorem:GVSSoundAndComplete}
The logic \il is sound and complete w.r.t.~GVS.
\end{theorem}

\begin{proof}
Soundness follows from an easy check on the rules and all the axiom schemes. For completeness we reason as follows. Suppose $\il \nvdash A$. By using de Jongh and Veltman's theorem from \cite{JonghVeltman:1990:ProvabilityLogicsForRelativeInterpretability} we get a (regular) Veltman model $\mathfrak M = \langle W, R, \{ S_w : w\in W\},V \rangle$ and world $x\in W$ so that $\mathfrak M , x \Vdash \neg A$. 

We now transform $\mathfrak M$ into a generalised Veltman model $\mathfrak M'$ by only changing the $S_w$ relations so that $\mathfrak M' = \langle W, R, \{ S'_w : w\in W\},V \rangle$ and\footnote{We would like to emphasise that we use the letter $V$ to denote both the valuation and a subset of $W.$ 
In the following text, we will do this again several times. } 
\[
uS'_w V \ :\Leftrightarrow \ \exists \, v{\in}V\ uS_wv.
\]
It will be clear from the context what the letter $V$ means.

Quasi-transitivity requires a small argument but it is easy to see that this definition of $S'$ yields a generalised Veltman model. Furthermore, via an easy induction we can prove that for any formula $B$ and any $w\in W$ we have $\mathfrak M, w \Vdash B \Leftrightarrow \mathfrak M', w \Vdash B$. In particular $\mathfrak M', x \Vdash \neg A$ which completes the proof.
\end{proof}

The above proof tells us that any Veltman model can be transformed into a generalised Veltman model preserving truth. Verbrugge has proven that in certain cases, one can also go the other way around and we will discuss this in Subsection \ref{section:VeltmanSemanticsVersusGVS}.

\subsection{On quasi-transitivity}
As we mentioned before, there are quite some alternatives to a semantic requirement of the transitivity axiom scheme $(A\rhd B) \wedge (B\rhd C) \to A\rhd C$. We will now discuss some of them. In the next table, we should bear in mind that if we have $uS_xV$ for some $u,x$ and $V$, then this automatically implies that $V\neq \varnothing$.


\begin{table}[H]
\label{table:quasi-trans}
\centering
\scriptsize
\begin{tabular}{c|l|l}
Nr. & Semantic requirement for transitivity & First mentioned in\\
\hline
(1) & \(uS_xY \Rightarrow  \forall  \, \{ Y_y\}_{y\in  Y} \Big((\forall \, y\in Y\ yS_xY_y) \Rightarrow  \exists  Z\subseteq  \bigcup_{y\in  Y}Y_y \  uS_xZ\Big)\) & This paper\\
(2) & \(uS_xY \Rightarrow  \forall  \, \{ Y_y\}_{y\in  Y} \Big((\forall \, y\in Y\ yS_xY_y) \Rightarrow  uS_x\bigcup_{y\in  Y}Y_y\Big)\) & Verbrugge '92  \cite{Verbrugge}\\
(3) & \(uS_xY \Rightarrow  \exists \, y\in Y\, \forall  Y'(yS_xY' \Rightarrow  \exists  \, Y''{\subseteq }Y' \ uS_xY'')\) & This paper\\
(4) & \(uS_xY \Rightarrow  \exists \, y\in Y\, \forall  Y'(yS_xY' \Rightarrow  uS_xY')\) & Joosten '98 \cite{Joosten:1998:MasterThesis}\\
(5) & \(uS_xY \Rightarrow  \forall \, y\in Y\, \forall  Y'(yS_xY' \Rightarrow  \exists  \, Y''{\subseteq }Y' \  uS_xY'')\) & This paper\\
(6) & \(uS_xY \Rightarrow  \forall \, y\in Y\, \forall  Y'(yS_xY' \Rightarrow  uS_xY')\) & Verbrugge '92 \cite{Verbrugge}\\
(7) & \(uS_xY \Rightarrow  \forall \, y\in Y\, \forall  Y'(yS_xY' \ \& \ y\notin Y' \Rightarrow  \exists  \, Y''{\subseteq }Y'\ uS_xY'')\) & This paper\\
(8) & \(uS_xY \Rightarrow  \forall \, y\in Y\, \forall  Y'(yS_xY' \ \& \ y\notin Y' \Rightarrow  uS_xY')\) & Goris, Joosten '09 \cite{GorisJoosten:2011:ANewPrinciple}\\

\end{tabular}
\caption{\label{tab:transitivity} Semantic conditions for quasi-transitivity mentioned in the literature.}
\end{table}

All of the presented quasi-transitivity requirements are adequate for proving
\il soundness  and completeness. For soundness it is routine to check that every
instantiation of ${\sf J2}$ holds. For the completeness part it is enough to see
that any ordinary Veltman model \(\mathfrak{M}=\langle W,R,\{ S_w: w\in W\}, V\rangle \)
can be transformed into a
generalised Veltman model \(\mathfrak{M}'=\langle W,R,\{ S'_w: w\in W\},V\rangle \)
where for all $w\in W$ we have \({S'_w\coloneqq \{\langle x,\{y\}\rangle :\langle x,y\rangle \in S_w\}}\)
and see\footnote{The argument works in virtue that in (1) and (2) we understand that each $y$ has a single $Y_y$ and would break down in case we allowed for various $Y_y^{(\alpha)}$ (yielding yet another semantic variation).} that
 $\mathfrak{M}'$ has the same associated forcing relation as $\mathfrak{M}$.
These properties have been verified in the proof assistant Agda and are presented with more detail in \cite{MasRovira:2020:MastersThesis}.

Agda (\cite{norell:thesis}) is a proof assistant based on a constructive type theory with dependent types that allows the paradigm of \textit{propositions as types} (\cite{wadler2015propositions}) via the Curry-Howard correspondence.

The following theorem tells us why Notion $(2)$ is in a sense the more natural one.
\begin{theorem} \label{org30d55aa} Let \(\mathfrak{F}=\langle W,R,\{ S_w: w\in W\}\rangle \)
be a generalised Veltman frame
satisfying quasi-transitivity Condition $(i)\in \{1, \ldots, 8\}$. Let \(\mathfrak{F}'=\langle W,R,\{ S'_w: w\in W\}\rangle \)
where for all $w\in W$ we define \(S'_w\) as the
monotone closure of \(S_w\):
\[S'_w\coloneqq \{\langle x,Y'\rangle  : \langle x,Y\rangle \in S_w, Y\subseteq Y'\subseteq R[w]\}.\]
Then \(\mathfrak{F}'\) is a generalised Veltman frame
satisfying quasi-transitivity Condition (2).
Furthermore for any formula $A$ and valuation $V$ with $\mathfrak{M} \coloneqq \langle \mathfrak{F},V\rangle$
and $\mathfrak{M'} \coloneqq \langle \mathfrak{F'},V\rangle$ we have that
\[
\mathfrak{M},w\Vdash A\ \mbox{ if and only if } \ \mathfrak{M}',w\Vdash A \ .
\]
\end{theorem}

\begin{proof}
Details are presented in \cite{MasRovira:2020:MastersThesis}.
\end{proof}


As we see in Theorem  \ref{org30d55aa} taking the monotone closure 
of each $S_w$ does not change the forcing relation and 
the resulting frame satisfies quasi-transitivity Condition (2). 

Note that taking the monotone closure of each $S_w$ is essentially different than
assuming that each $S_w$ is monotone by definition, as then the forcing relation may change. In the following example we present a generalised Veltman model with Condition (8) that showcases such behaviour.

\begin{figure}[H]
\centering
\includegraphics[width=0.3\textwidth]{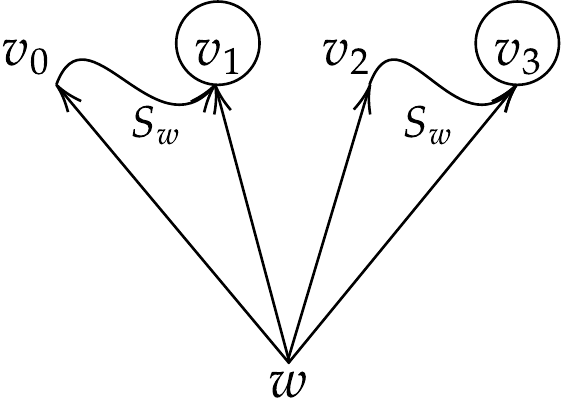}

\caption{\label{fig:org1323f27}Example frame: \(wRv_0,wRv_1,wRv_2,wRv_3\), \(v_0S_w\{v_1\}\), \(v_2S_w\{v_3\}\).}
\end{figure}
\noindent

Let \(\mathfrak{M}\) be a model based on the frame displayed in Figure \ref{fig:org1323f27}
such that $V(p) = \{v_0\}$ and $V(q) = \{v_2\}$ (i.e.\ ${\{x:x\Vdash p\}  = \{v_0\}}$, \(\{x:x\Vdash  q\}  = \{v_2\}\)).
We see that \(w\Vdash \neg (p \rhd  q)\) as
\(p\) is only true in \(v_0\) and we only have \(v_0S_w\{v_1\}\) with \(v_1\nVdash q\). If we assume that the relation \(S_w\) is monotone then we have \(v_0S_w \{v_1, v_2\}\) and by
quasi-transitivity (8) we get \(v_0S_w \{v_3\}\). Consequently
\(w\Vdash \neg (p \rhd  q)\) is no longer true.

\subsection{Veltmans semantics versus GVS}\label{section:VeltmanSemanticsVersusGVS}

The completeness proof of GVS (Theorem \ref{theorem:GVSSoundAndComplete}) tells us that any Veltman model can be transformed into a generalised Veltman model preserving truth. Verbrugge has proven that one can also go the other way around for quasi-transitivity Notion 6 from Table \ref{tab:transitivity}. Below we write $\Vdash$ and $\Vdash'$ instead of $V$ and $V'$ and their respective extensions.
\newcommand{\hookdoubleheadrightarrow}{%
  \hookrightarrow\mathrel{\mspace{-15mu}}\rightarrow
}
\begin{theorem}(Verbrugge \cite{Verbrugge})\label{theorem:FromGeneralisedToRegular}
Let $\la W, R, \{ S_w: w\in W\}, V \ra$ 
be a generalised Veltman model with quasi-transitivity Condition $(i)\in\{3,4,5,6\}$ (see Table \ref{table:quasi-trans}). There is a (regular) Veltman model $\la W', R', \{ S'_w: w\in W'\}, V' \ra$
and a map $f : W \to \wp (W')$ so that for each $w\in W$ and each $w'\in f(w)$ we have for any formula $B$ that 
\[
\mathfrak{M}, w \Vdash B \ \ \mbox{ if and only if } \ \ \mathfrak{M'}, w' \Vdash' B.
\]
Here $\Vdash$ is the forcing relation in $\mathfrak{M}$ based on $V$ and $\Vdash'$ is the forcing relation in $\mathfrak{M}'$ based on $V'$.
\end{theorem}

\begin{proof}
We refer the reader to \cite{MasRovira:2020:MastersThesis}
for details and mainly present the definition of $\mathfrak{M}'$ here as was given in \cite{Verbrugge}. We will define a regular Veltman model $\mathfrak{M}'$ out of generalised Veltman model $\mathfrak{M}$. The main idea is that we will take many copies of worlds in $\mathfrak{M}$. When we define some $x'S'_wy'$ we should take into account that the single worlds $x'$ and $y'$ from $\mathfrak{M}'$ somehow come from worlds $x$ and $y$ from $\mathfrak{M}$ where these $x$ and $y$ fulfilled many roles as elements of images of the $S$ relation. To capture this richness of the generalised Veltman semantics, we shall choose some representatives from $S$ images. To this end, we first define for every world $x \in W$ a set ${\sf SR}(x)$ which contains all sets which are so-called \emph{$S$-representatives} in a sense that whenever $xS_uV$, then any $S$-representatives of $x$ will mention some non-zero number of elements of $V$. The formal definition reads as follows:
\[
\begin{array}{ll}
{\sf SR}(x) \ := \Big{\{} A\subset W\times W \mid  & \forall u \forall \, V {\subseteq} W \big( xS_uV \Rightarrow \exists\, v{\in} V \ \la u,v\ra \in A\big) \ \& \\
 & \forall u, v\ \Big( \la u,v\ra \in A \Rightarrow \exists \ V {\subseteq} W \big( xS_uV \ \& \ v\in V\big)\Big) \Big{\}}\ .
\end{array}
\]
We observe that a world $x$ will typically have many $S$-representatives. In the new model, we will consider all of them. Thus, we can now define the domain as
\[
W':= \{ \la x , A \ra \mid A{\in}{\sf SR}(x) \mbox{ or }  {\sf SR}(x) = \varnothing = A\}.\\
\]
To conclude, the relations are defined as 
\[
\begin{array}{rcl}
\la x,A\ra R' \la y,B \ra  &:\Leftrightarrow&  xRy \ \& \ \forall w,z  (wRx\ \& \ \la w,z\ra \in B \Rightarrow  \la w,z\ra \in A);\\
\mbox{}\\
\la x, A\ra S'_{\la w,C\ra} \la y,B \ra & :\Leftrightarrow & 
\la w,C\ra R'\la x, A\ra \ \& \  \la w,C\ra R'\la y, B\ra \ \& \\
 & &
 \mbox{}\hskip 2em \forall v\ \big( \la w,v \ra \in B \Rightarrow \la w,v \ra \in A \big);
\end{array}
\]
and finally $\la x,A\ra \Vdash p \ :\Leftrightarrow \ x \Vdash p$. Verbrugge proved that $\mathfrak{M}'$ indeed defines a regular Veltman model and that moreover, for each formula $A$, for each world $x$ and for each $V\subseteq W$ so that $\la x, V \ra \in W'$ we have $x\Vdash A \ \Leftrightarrow \ \la x,V \ra \Vdash' A$.

\end{proof}

Verbrugge showed the above theorem for generalised Veltman models with quasi-transitivity Condition (6). We have slightly improved the result by showing that it also holds for Conditions (3), (4) and (5). The above proof, together with a substantial simplification, has been fully formalised in the proof assistant Agda and is presented in \cite{MasRovira:2020:MastersThesis}. 
Vukovic studies\footnote{The proof in \cite{Vukovic08} contains a minor typo/error and in \cite{MasRovira:2020:MastersThesis} this is addressed.} in \cite{Vukovic08} how obtaining a Veltman model from a Generalised Veltman model can be performed for the, by now standard, transitivity Condition 2.
\ \\
\medskip

The above observations tell us that when it comes to models, regular Veltman semantics and generalised Veltman semantics are equally powerful. With respect to frames the panorama is very different. Before we make this precise, let us first discuss frame conditions for GVS.

Let \principle{X} be a modal axiom scheme. We denote by \kgen{X} a formula of first-order or 
higher-order logic  such that for all generalised Veltman frames $\mathfrak{F}$ the following 
holds:
$$  \mathfrak{F}\Vdash \mathsf{X} \ \mbox{ if and only if } \ 
    \mathfrak{F}\models \kgen{X} .$$
The formula \kgen{X} is called characteristic property (or frame condition) 
of the given logic \il{X}. 
The class of all generalised Veltman frames $\mathfrak{F}$ such that 
$\mathfrak{F}\models \kgen{X}$ is called the 
characteristic class of generalised frames for \il{X}.
If $\mathfrak{F}\models \kgen X$ we also say that the frame $\mathfrak F$ possesses the property 
\kgen{X}. We say that a generalised Veltman model $\mathfrak{M}=\la W,R,\{ S_w:w\in W\},V\ra$ 
is an \ilgen{X}-model, or that model $\mathfrak{M}$ possesses the property \kgen{X}, if 
the frame $\la W,R,\{S_w:w\in W\}\ra$ possesses the property \kgen{X}. 

Vukovi\'{c} \cite{Vukovic08} studied in a general setting how to transform a generalised Veltman model to an ordinary 
Veltman model much in the spirit of Theorem \ref{theorem:FromGeneralisedToRegular}, such that these two models are bisimilar (in some aptly defined sense). Such a program can only yield partial answers w.r.t.~frames since, as we shall see in Section \ref{section:ModalCompleteness},
the logic \extil{P_0} is complete w.r.t.\ generalised semantics, but incomplete w.r.t.\ ordinary semantics.

\section{Generalised Veltman semantics for separating systems}

In this section we briefly mention some results where GVS has been used to prove independence of various systems. 
However, what makes GVS really interesting in our opinion, are its good 
model-theoretical properties. We will discuss those in later sections.

\subsection{Principles and Veltman models}
\label{sec:principles_models}
In  \cite{Visser:1990:InterpretabilityLogic}, Visser studies among others relations between 
various extensions of the basic interpretability logic \il. Among others, he considered the 
following principles: 
\[
\begin{array}{rll}
\principle{W} &:= & A\rhd B\to A\rhd B\wedge\Box\neg A ;\\
\principle{KW1}& := & A\rhd \Diamond \top\rightarrow \top \rhd \neg A; \\
\principle{F}&:= & A\rhd\Diamond A\to\Box\neg A.
\end{array}
\]
Visser observed\footnote{Even though indeed \principle{KW1} is similar in flavour, it turned out that its frame-condition is actually, contrary to what Visser thought and so announced in \cite{Visser:1990:InterpretabilityLogic}, different from that of \principle{W}. \v{S}vejdar computed and published the corrected condition in \cite{svej91}: for each $wRy$ there exists $x{\in} \mathrm{M}(w)$ such that $yS_w x,$  where $\mathrm{M}(w):=\{ x{\in} R[w] :$ there is no $z{\in} R[w]$ such that $x\,(S_w{\circ} R)\, z\}.$}
that all of \principle{W}, \principle{F} and \principle{KW1} define proper 
extensions of \il that have the same frame condition w.r.t.~Veltman semantics: for each $w$, the 
relation $R\circ S_w$ should be conversely well-founded. Further, he noted that $\extil{W}\vdash 
\principle{KW1}$ and $\extil{W}\vdash \principle{F}$ (already in 
\cite{Visser:1988:preliminaryNotesOnInterpretabilityLogic}) and he posed as an open question 
if the converse also holds (over \il). As a mere curiosity it was mentioned that a slight 
weakening of \principle{F} does not yield any extension of \il. We repeat that here: if we 
take the contraposition $\Diamond A \to \neg (A\rhd \Diamond A)$ of \principle{F} and replace 
the implication by an interpretability modality we obtain an \il provable 
formula\footnote{Principle \principle{K10} in 
\cite{Visser:1988:preliminaryNotesOnInterpretabilityLogic}.}: $\Diamond A \rhd \neg 
(A \rhd \Diamond A)$. 

Another family of principles studied in \cite{Visser:1988:preliminaryNotesOnInterpretabilityLogic} 
is given by:
\[
\begin{array}{rll}
\principle{M}&:= & A\rhd B\to A\wedge\Box C\rhd B\wedge\Box C;\\
\principle{KM1}& := & A\rhd \Diamond B\rightarrow \Box (A\rightarrow \Diamond B);\\ 
\principle{KM2}& 
:=& A\rhd B\rightarrow \big(\Box (B\rightarrow \Diamond C)\rightarrow \Box (A\rightarrow 
\Diamond C)\big).
\end{array}
\]

It was observed that all of \principle{M}, \principle{KM1} and \principle{KM2} have the same frame 
condition w.r.t.~Veltman semantics: $yS_wzRu \ \Rightarrow \ yRu$. Similar to the previous family, it was 
observed that $\ilm\vdash \principle{KM1}, \principle{KM2}$ and posed in 
\cite{Visser:1988:preliminaryNotesOnInterpretabilityLogic,Visser:1990:InterpretabilityLogic} as an 
open question if the converse also holds (over \il). For this family it was proven in 
\cite{Visser:1988:preliminaryNotesOnInterpretabilityLogic}  that \principle{KM1} and 
\principle{KM2} are interderivable over \il. Moreover, just as \principle{W} follows from 
\extil{M}, we also have that \principle{KW1} follows from \extil{KM1}.

\v{S}vejdar in 1991 took up the above mentioned questions of Visser's whether certain reversals 
like $\extil{KM1}\vdash \principle{M}$ hold. Ordinary Veltman models were suitable to distinguish 
all combinations of the following principles 
of interpretability \cite{svej91}:
\principle{W}, \principle{M}, \principle{KM1}, \principle{KW1}, 
\principle{KW1^0}, and \principle{F}, where 
\[
\principle{KW1}^0 \ := \ A\wedge B\rhd\Diamond A \rightarrow A 
\rhd(A\wedge \neg B).
\]
Unlike most other proofs of independence results which rely on differences in characteristic 
classes, parts of his proofs are based on exhibiting particular \textit{models} that globally
satisfy one of the principles in question. 
For example, his proof that $\extil{\{F, KW1\}} \nvdash \principle{KW1}^0$ exhibits an 
\extil{F}-model that globally satisfies \principle{KW1}, but refutes $\principle{KW1}^0$.
As a consequence, \v{S}vejdar established that the logics \extil{F}, \extil{KW1} and \extil{KM1} 
are incomplete with respect to their class of frames. 

\subsection{Generalised frame conditions and independence}
\label{section:Independence}

In the previous subsection we saw various principles having the same frame condition. However, their frame conditions for GVS differ and as such this provides a way of telling different logics apart. In this section we will simply present a collection of generalised frame conditions and leave it as an easy exercise that they are all different from each other.
Verbrugge \cite{Verbrugge} determined \kgen M, \kgen{KM1}, and \kgen P:
$$\begin{array}{rcl}
\kgen M & := &  uS_w V \Rightarrow (\exists V' \subseteq V)( uS_w V' \ \& \ R[V'] 
\subseteq R[u]); \\
\mbox{}\\
\kgen {KM1} & := & uS_w V \Rightarrow (\exists v\in V) \forall z (vRz \Rightarrow uRz);\\
\mbox{}\\
\kgen P & := & wRw'RuS_w V \Rightarrow (\exists V' \subseteq V)\ uS_{w'} V'. 
\end{array}$$
She proved 
$\extil{KM1}\nvdash \principle{M}$, 
$\extil{F}\nvdash \principle{W}$, and
$\extil{F}\nvdash \principle{KW1}$ using GVS. 

There are two more principles that frequently occur in the literature. First, there is 
\[
\principle{M}_0 \ := \ A\rhd B\to\Diamond A\wedge\Box C\rhd B\wedge\Box C.
\]
And second, there is
\[
\principle{W}^*\ := \ A\rhd B\rightarrow B\wedge \Box C\rhd B\wedge \Box C \wedge \Box \neg A .
\]
Visser showed in \cite{Visser:1991:FormalizationOfInterpretability} that $\extil{W}\not\vdash \principle{M}_0$ and that $\extil{M_0W} = \extil{W^*}$. 

Vukovi\'{c} in \cite{Vukovic96} determined the formula $\kgen{M$_0$}:$ 
$$\kgen{M$_0$}\ := \ wRuRxS_wV\ \Rightarrow \ (\exists V'\subseteq V) 
(uS_wV' \ \& \ R[V']\subseteq R[u]).$$
and proved independence of the principle $\principle{M}_0$ with various others principles of 
interpretability.
All connections between principles \principle{M}, \ \principle{M}$_0,$ \ \principle{KM1}, \ 
\principle{KM2}, \ \principle{P}, \ \principle{W}, \ \principle{W}$^*$, \ \principle{KW1}$^0,$ \ 
\principle{KW1} \ and \ \principle{F} were determined in \cite{Vukovic99} using GVS. 
Vukovi\'{c} provided in \cite{Vukovic99} a comparative modal study of all these principles 
together using GVS. The result of this study can be summarized by the following diagram:

\vskip 2ex
\begin{picture}(120,60)
\put(26,54){\principle{P}} 
\put(27,52){\vector(0,-1){13}}

\put(44,54){\principle{M}} 
\put(44,52){\vector(-1,-1){13}}
\put(48,52){\vector(1,-1){13}}

\put(5,35){\principle{M_0}}
\put(23,36){\vector(-1,0){12}}

\put(25,35){\principle{W^\ast}} 
\put(27,33){\vector(0,-1){13}}

\put(58,35){\principle{KM1}} 
\put(67,36){\vector(1,0){13}}
\put(61,33){\vector(-1,-1){13}}

\put(82,35){\principle{KM2}} 
\put(79,36){\vector(-1,0){12}}

\put(25,16){\principle{W}} 
\put(30,17){\vector(1,0){13}}

\put(44,16){\principle{KW1^0}} 
\put(46,14){\vector(-1,-1){11}}
\put(50,14){\vector(1,-1){11}}

\put(33,0){\principle{F}} 
\put(58,0){\principle{KW1}}
\end{picture}

\vskip 3ex
\noindent
Joosten and Visser presented a new \ilal principle
\[
\principle{P}_0 \ :=  \ A\rhd\Diamond B\to\Box(A\rhd B)
\]
in Joosten's master thesis \cite{Joosten:1998:MasterThesis}. Using GVS but without establishing 
the frame condition for \principle{P_0} Joosten could prove that $\principle{W}, \principle{M_0}$ and 
$\principle{P_0}$ are maximally independent (no two imply the other).

Goris and Joosten considered the principle 
\principle{P_0} in \cite{GorisJoosten:2011:ANewPrinciple} and presented a related new principle that has the same frame condition
\[
\principle{R}\ := \  A\rhd B\to \neg(A\rhd\neg C)\rhd  B\wedge\Box C.
\]
They determined formulas \kgen{P$_0$} and \kgen{R}.
Here are slightly reformulated versions from \cite{Mikec-Vukovic-20}:
$$\begin{array}{rcl}
\kgen{P$_0$} & := &  wRxRuS_w V \ \& \ (\forall v \in V)  R[v] \cap Z \neq \emptyset \ 
\Rightarrow \ (\exists Z' \subseteq Z) uS_x Z';\\
\mbox{}\\
\kgen{R} & := & wRxRuS_w V \Rightarrow (\forall C \in \mathcal{C}(x, u))(\exists U \subseteq V)
(x S_w U \ \& \ R[U] \subseteq C).
\end{array}$$
where $\mathcal{C}(x, u) = \{ C \subseteq R[x] : (\forall Z) (uS_x Z\Rightarrow Z\cap C \neq 
\emptyset) \}$ is the family of ``choice sets''.
They proved 
$\extil{}\principle{W}\principle{P}_0\principle{M}_0\not\vdash \principle{R}$  using GVS. 
Instead of providing a GVS frame condition for \principle{W}, the authors proved a necessary and 
sufficient GVS frame condition for \principle{W} to fail. 

\begin{definition}
\[
\label{defn:not_w}
\begin{array}{ll}
{\sf Not\text{-}W} \ \ :=\ \  & \exists \,  w,\  z_0,\  \{Y_i\}_{i\in \omega}, \ \{ y_i\}_{i\in \omega,\  y_i\in Y_i} ,\  Z,\   
\{  z_{i+1}\}_{i \in \omega, \ z_{i+1}\in Z}\\
\ & [ \forall i \in \omega (z_i S_w Y_i \ni y_i Rz_{i+i}) \ \& \\
\ & \ \forall z\in Z \exists i\in \omega zS_w Y_i \ \& \\
\ & \ \forall z\in Z \forall Y\ (zS_wY \ \& \ Y \subseteq (\cup_{i\in \omega}Y_i) \ \Rightarrow \  
\exists z' \in Z \exists y \in Y yRz')]\\

\end{array}
\]
\end{definition}

\begin{lemma}
For any generalised Veltman frame $\mathfrak F$ we have that
\[
\mathfrak{F} \models {\sf Not\text{-}W} \ \ \ \mbox{if and only if} \ \ \ \mathfrak{F} \not \models \mathsf W.
\]
\end{lemma}
A positive frame condition for \principle{W} is presented in Section \ref{section:LogicsWandStar}.

\section{Modal completeness: preliminaries}\label{section:CompletenessPreliminaries}

The aim of this and the next section is to explore modal completeness with respect to GVS. We will give the state-of-the-art of completeness results involving GVS.
Let us first say a few words on the history of modal completeness proofs concerning interpretability logics.

This and the following section are based heavily on the recent paper \cite{Mikec-Vukovic-20}.
For this reason we will not cite results; for any definition or result without a reference it is safe to assume it is being quoted from \cite{Mikec-Vukovic-20}.

\subsection{Overview of approaches}

De~Jongh and Veltman proved the completeness of \extil{}, \extil{M} and \extil{P}  w.r.t.\ 
the corresponding characteristic classes of ordinary (and finite) Veltman frames in 
\cite{JonghVeltman:1990:ProvabilityLogicsForRelativeInterpretability}.
As is usual for extensions of the provability logic \textbf{GL}, all completeness proofs 
suffer from compactness-related issues. 
One way to go about this is to define a (large enough) adequate set of formulas and let worlds 
be maximal consistent subsets of such sets (used e.g. in 
\cite{JonghVeltman:1990:ProvabilityLogicsForRelativeInterpretability}). 
With interpretability logics and 
ordinary Veltman semantics, worlds have not been identified with (only) sets of formulas. It seems 
that with ordinary Veltman semantics it is sometimes necessary to duplicate worlds (that is, have 
more than one world correspond to a single maximal consistent set) in order to build models for 
certain consistent sets (see e.g.\ 
\cite{JonghVeltman:1990:ProvabilityLogicsForRelativeInterpretability}). In  
\cite{JonghVeltman:1999:ILW}, de Jongh and Veltman proved completeness of the logic \extil{W}  
w.r.t.\ its characteristic class of ordinary (and finite) Veltman frames.

Goris and Joosten, inspired by Dick de Jongh, introduced\footnote{See our comments in Footnote \ref{footnoteOnSteps} for some more detailed historical context. Also \cite{de2004completeness} provides some comments on construction methods for other modal logics.} a more robust approach to proving completeness of interpretability 
logics, the \textit{construction method} or \textit{step-by-step method} (\cite{Goris-Joosten-08, GorisJoosten:2011:ANewPrinciple}). 
In this type of proofs, one builds models step by step, and the final model is retrieved 
as a union. 
While closer to the intuition and more informative than the standard proofs, these proofs are 
hard to produce and verify due to their size. (They might have been shorter if tools from 
\cite{BilkvaGorisJoosten:2004:SmartLabels, goris2020assuring} have 
been used from the start.) 
For the purpose for which this type of proofs was invented (completeness of \extil{M_0}{} and 
\extil{W^*} w.r.t.\ the ordinary semantics), this type of proofs is still the only known approach that works. 

In \cite{Mikec-Vukovic-20} a very direct type of proofs of completeness is presented; similar to 
\cite{JonghVeltman:1990:ProvabilityLogicsForRelativeInterpretability} in the general approach,
but this time with respect to GVS.
The so-called \textit{assuring labels} from 
\cite{BilkvaGorisJoosten:2004:SmartLabels, goris2020assuring} were used as a key step.
These completeness proofs are the ones that we aim to explore here.
An example that illustrates benefits of using the generalised semantics will be given in the section dedicated 
to \extil{M_0}. The most interesting of these results are completeness of
\extil{R} and \extil{P_0}. The principle \textsf{R} is important because it forms the basis of the, at 
the moment, best explicit candidate for \extil{(All)} as discussed in more detail in Section \ref{section:Hierarchies}. Results concerning the principle \extil{P_0} 
are interesting in a different way; they answer an old question: is there an unravelling technique that 
transforms generalised \extil{X}-models to ordinary \extil{X}-models, that preserves satisfaction of relevant 
characteristic properties? The answer is \textit{no}: \extil{P_0} is complete w.r.t.\ GVS, but it is known to be incomplete w.r.t.\ the ordinary semantics (\cite{GorisJoosten:2011:ANewPrinciple}).

\subsection{Completeness w.r.t.\ generalised semantics}

In what follows, ``formula'' will always mean ``modal formula''. If the ambient logic in some context is 
\extil{X}, a maximal consistent set w.r.t.\ \extil{X} will be called an \extil{X}-MCS. Let us now introduce 
\textit{assuring labels} from \cite{BilkvaGorisJoosten:2004:SmartLabels} and \cite{goris2020assuring}.

\begin{definition}[\cite{BilkvaGorisJoosten:2004:SmartLabels}, a slightly modified Definition 3.1]
    \label{def-smart-label}
    Let $w$ and $u$ be some \extil{X}-MCS's, and let $S$ be an arbitrary set of 
    formulas.
    We write $w \prec_S u$ if for any finite $S' \subseteq S$ and any formula $A$ 
    we have that $A \rhd \bigvee_{G\in S'}\neg G \in w$ implies 
    $\neg A, \square \neg A \in u.$
\end{definition}
Note that the small differences between our Definition \ref{def-smart-label} and Definition 3.1 
\cite{BilkvaGorisJoosten:2004:SmartLabels} do not affect the results of \cite{BilkvaGorisJoosten:2004:SmartLabels} that we use.\footnote{
The difference is a different strategy of ensuring converse well-foundedness for the relation $R$. 
Instead of asking for the existence of some $\Diamond F \in w \setminus u$ whenever $w R u$, as is usual in the 
context of provability (and interpretability) logics, we will go for a stronger condition (see Definition 
\ref{ilx-struktura}).
Since we will later put $R := \prec$, this choice of ours is reflected already at this point.
}

\begin{definition}[\cite{BilkvaGorisJoosten:2004:SmartLabels}, page 4]
    \label{def-smart-label-sets}
    Let $w$ be an \extil{X}-MCS, and $S$ an arbitrary set of formulas. Put:
    \begin{align*}
    	w_S^\square & := \{ \square \neg A : \exists S' \subseteq S, S' \text{ finite}, 
    	A \rhd\bigvee_{G\in S'}  \neg G\in w \}; \\
       	w_S^\boxdot & := \{ \neg A, \square \neg A : \exists S' \subseteq S, S' \text{ finite}, 
       	A \rhd\bigvee_{G\in S'}  \neg G\in w \}.
   \end{align*}
\end{definition}
Thus, $w \prec_S u$ if and only if $w_S^\boxdot \subseteq u$.
If $S = \emptyset$ then $w_\emptyset^\square = \{ \square \neg A : A \rhd \bot \in w \}$. 
Since $w$ is maximal consistent, use of $w_\emptyset^\Box$ usually amount to the same as the use of the set
$\{ \square A : \square A \in w \}.$ 

We will usually write $w \prec u$ instead of $w \prec_\emptyset u$.
\begin{lemma}[\cite{BilkvaGorisJoosten:2004:SmartLabels}, Lemma 3.2]
\label{lema3.2}
Let $w$, $u$ and $v$ be some \extil{X}-MCS's, and let $S$ and $T$ be some sets of formulas. Then we have:
\begin{itemize}
    \item[a)] if $S\subseteq T$ and $w\prec_T u$, then $w\prec_S u$;
    \item[b)] if $w\prec_S u\prec v$, then $w\prec_S v$;
    \item[c)] if $w\prec_S u$, then $ S\subseteq u$.
\end{itemize}
\end{lemma}
We will tacitly use the preceding lemma in most of our proofs. Although not needed in this paper, we mention that in \cite{goris2020assuring} it is shown that without loss of generality we may actually assume that labels are full theories.

The following two lemmas can be used to construct (or in our case, find) a MCS with the required properties.
\begin{lemma}[\cite{BilkvaGorisJoosten:2004:SmartLabels}, Lemma 3.4]
\label{problemi}
Let $w$ be an \extil{X}-MCS, and let $\neg (B\rhd C)\in w.$ Then there is an \extil{X}-MCS $u$ 
such that $w \prec_{\{\neg C\}} u$ and $B, \square \neg B\in u.$
\end{lemma}
\begin{lemma}[\cite{BilkvaGorisJoosten:2004:SmartLabels}, Lemma 3.5]
\label{nedostaci}
Let $w$ and $u$ be some \extil{X}-MCS's such that $B\rhd C\in w,$ $w\prec_S u$ and $B\in u.$
Then there is an \extil{X}-MCS $v$ such that $w\prec_S v$ and $C,\square\neg C\in v.$
\end{lemma}

In the remainder of this section, we will assume that $\mathcal{D}$ is always a finite set of formulas, 
closed under taking subformulas and single negations, and $\top\in \mathcal{D}$.
The following definition is central to most of the results of this section.

\begin{definition}
\label{ilx-struktura}   
Let \textsf{X} be a  subset of $\{$\textsf{M}, \textsf{M\textsubscript{0}}, \textsf{P}, 
\textsf{P\textsubscript{0}}, \textsf{R}$\}$.
We say that  
$\mathfrak{M} =\la W, R, \{S_w : w \in W\}, V\ra$ is \emph{the \extil{X}-structure for 
a set of formulas $\mathcal{D}$} if:
\[
\begin{array}{rll}
           W &:=& \{ w : w \text{ is an \extil{X}-MCS and for some } G \in \mathcal{D}, G 
           \wedge \square \neg G \in w \};\\
            wRu &:\Leftrightarrow & w \prec u;\\  
            uS_w V &:\Leftrightarrow & wRu \mbox{\ and, } V \subseteq R[w]\mbox{ and, } (\forall S)(w \prec_S u \Rightarrow 
            (\exists v \in V) w \prec_S v );\\
            w\in V(p) &: \Leftrightarrow & p\in w.
    \end{array}
\]
\end{definition}

We note that the \extil{X}-structure for $\mathcal{D}$ is a unique object. In fact, we could work with just one ``\extil{X}-structure'' (that would not depend even on $\mathcal{D}$): the disjoint union of \extil{X}-structures for all choices of $\mathcal{D}$. We also observe that the definition entails that when $uS_wV$, then $V\neq \emptyset$ since $wRu \Rightarrow w\prec_\emptyset u$ so $\exists (v \in V) w\prec_\emptyset v$.

Notice that worlds in the definition above are somewhat more restricted than what is usually found in similar proofs: every world is required to be $R$-maximal
with respect to some formula.
That is, for every world $w \in W$ we want to have a formula $G_w$ such that $w \Vdash G_w$ and for any $R$-successor $u$ of $w$, $u \nVdash G_w$. This is equivalent to the requirement that for some formula $G_w$, $w \Vdash G_w \wedge \Box \neg G_w$.
Of course, before we prove our truth lemma we can only require that $G_w \wedge \Box \neg G_w \in w$.
Because of this we need the following lemma whose proof boils down to an instance of Löb's axiom.

\begin{lemma}
\label{lema-nepraznost}
If $\extil{X} \nvdash \neg A$ then there is an \extil{X}-MCS $w$ such that $A\wedge \square\neg A\in w.$
\end{lemma}

We are now ready to prove the main lemma of this section,
    which tells us that the structure defined in
    Definition \ref{ilx-struktura} really is a generalised
    Veltman model.
Notice that we do not claim that it is also an
    \ilgen{X}-model;
    we prove that later.

\begin{lemma}
Let \textsf{X} be a subset of $\{$\textsf{M}, \textsf{M\textsubscript{0}}, \textsf{P}, 
\textsf{P\textsubscript{0}}, \textsf{R}$\}$.
The \extil{X}-structure $\mathfrak{M}$ for a set of formulas $\mathcal{D}$ is a generalised Veltman model.
Furthermore, the following truth lemma holds:
\[ 
    \mathfrak{M},w\Vdash G \ \mbox{ if and only if }\ G\in w,
\]
for all $G\in\mathcal{D}$ and $w\in W.$ 
\label{lemma-main}
\end{lemma}
\begin{proof}
Most of the proof is straightforward.
Let us just comment the proof of the truth lemma,
    more specifically, the following claim in the induction step:
    $w \Vdash B \rhd C \Rightarrow B \rhd C \in w$.
This part is probably the most interesting one, since it explains
    the motivation behind the definition of $S_w$ in
    Definition \ref{ilx-struktura}.
    
Assume $B \rhd C \notin w$. Lemma \ref{problemi} implies there is $u$ with 
$w \prec_{\{\neg C\}} u$ and  $B, \square \neg B\in u$ (thus $u \in W$). 
It is immediate that $wRu$ and the induction hypothesis implies that $u \Vdash B$.
Assume $u S_w V.$ We are to show that $V \nVdash C$.
Since $w\prec_{\{ \neg C\}} u$ and $uS_w V$,  there is $v \in V$ 
such that $w \prec_{\{\neg C\}} v$. Lemma \ref{lema3.2} implies 
$\neg C\in v.$ 
The induction hypothesis implies $v \nVdash C$; thus $V \nVdash C$.
\end{proof}
This lemma is just one step away from a completeness proof:
\begin{theorem}
\label{glavni} 
Let $X \subseteq \{$\textsf{M}, \textsf{M\textsubscript{0}}, \textsf{P}, 
\textsf{P\textsubscript{0}}, \textsf{R}$\}$.
Assume that for every set $\mathcal{D}$ the \extil{X}-structure for $\mathcal{D}$ possesses 
the property \kgen{X}.
Then \extil{X} is complete w.r.t.\ \ilgen{X}-models.
\end{theorem}

\begin{proof}
Let $A$ be a formula such that $\il{X}\nvdash \neg A$. Lemma \ref{lema-nepraznost} implies there is 
an \extil{X}-MCS $w$ such that $A \wedge \square \neg A \in w.$ 
Let $\mathcal D$ have the usual properties, and contain $A$.
Let ${\mathfrak M}=\la W, R, \{S_w : w \in W\}, V\ra$ be the \extil{X}-structure for $\mathcal{D}$. 
Since $A \wedge \square \neg A \in w$  and $A \in \mathcal{D}$, we have $w \in W$.
Lemma \ref{lemma-main} implies $\mathfrak{M},w\nVdash \neg A.$ 
\end{proof}

\begin{corollary}
The logic \extil{} is complete w.r.t.\ GVS.
\end{corollary}

Note that any method for transforming generalised to ordinary models like
presented in Theorem \ref{theorem:FromGeneralisedToRegular} or in \cite{Vukovic08} now implies 
completeness of \extil{} w.r.t.\ ordinary Veltman models.


In the next section 
we comment on the completeness of the following logics w.r.t.\ GVS: \extil{M}, \extil{M\textsubscript{0}}, \extil{P}, \extil{P_0},  \extil{R}, \extil{W} 
and \extil{W^*}.

\subsection{A note on generalised Veltman semantics and labelling}

In all studied extensions of \extil{} we have to duplicate maximal consistent sets when building 
ordinary Veltman models for consistent sets of formulas. More accurately, no one seems to have 
come up with a natural way of assigning just one purpose to every maximal consistent set of 
formulas. For example, when building a model where $\{ \neg (p \rhd q), \neg(p \rhd r), p \rhd (q 
\vee r) \}$ is true in some world $w$, we could try to use the same set/world $u$ visible from $w$ as a witness for the formulas $ \neg (p \rhd q)$ and $\neg(p 
\rhd r)$ in $w$. For example, this may be the set where the only propositional formula is $p$, 
and no formula of form $\neg(A \rhd B)$ is contained. But, due to $p\rhd (q\vee r)$, in any model where $w$ is there we do require two worlds like $u$ within that model; one of which will 
have an $S_w$-successor satisfying $q$ but not $r$, and the other one an $S_w$-successor satisfying $r$ but not $q$.

GVS doesn't share this problem of duplication, at least not in any known case of a complete 
extension of \extil{}. A generalised model for the problem above is simple. Let $w = \{ \neg (p 
\rhd q), \neg(p \rhd r), p \rhd (q \vee r) \}$, $u =\{ p \}$, $x = \{ q \}$, $y = \{ r \} $, and 
let $wRuS_w \{ x, y \}$. Unspecified propositional formulas are assumed to be false, and 
unspecified $\rhd$-formulas are assumed to be true.

Now, having in mind this generalised model, what can be said about the $wRu$ transition in 
terms of labels? This might be important if we are building a generalised model step-by-step. 
Since $u$ has two roles, it would be natural to allow (even with assuringness) two labels: 
$\{\neg q\}$ and $\{\neg r\}$. And these labels are justified, since indeed 
$\{ x, y \} \nVdash q, r $.
Both these labels are expressible without sets (in terms of criticality, for example, the labels 
would be formulas $q$, and $r$, respectively). We recall that the $S_w$ from Definition \ref{ilx-struktura} indeed takes multiple labels into account.

However, there is another bit of label-related information that these facts do not express: which 
labels \textit{do not} hold. Although  $\{\neg q\}$ and $\{\neg r\}$ are justified choices, the 
label  $\{\neg q, \neg r\}$ is not a good choice. This label would require $\neg p \notin u$, which is 
clearly not the case. This is the information the assuringness allows us to express, and criticality 
does not.\footnote{
Granted, one might say that the inadequacy of the assuring label $\{\neg q, \neg r\}$ is 
equivalent to the inadequacy of the critical label $q \vee r$. 
However, expressing this fact in terms of criticality does not retain structural information 
of our situation; we see a disjunction where really we are only interested in disjuncts.}
Note that such a situation cannot happen in ordinary semantics: if the label 
$\{\neg q, \neg r\}$ is inappropriate for some $wRu$ transition, that means there is $A \in u$ 
with 
$A \rhd q \vee r \in w$. This, since we are now working in ordinary semantics, means there 
should be an $S_w$-successor of $u$ satisfying $q \vee r$. So, either this new world satisfies 
$q$ or $r$. 
So, $\{ \neg q \}$ or $\{ \neg r \} $ had to be inappropriate labels (for $wRu$) too.

\section{Modal completeness of various systems}\label{section:ModalCompleteness}

In this section we explore completeness proofs for various extensions of \extil{}.
We also briefly describe a recent preprint where certain subsystems of \extil{} are explored.

\subsection{The logic \extil{M}}

Completeness of the logic \extil{M} w.r.t.\ GVS is an easy consequence 
of the completeness of \extil{M} w.r.t.\ the ordinary semantics, first proved by de~Jongh and 
Veltman (\cite{JonghVeltman:1990:ProvabilityLogicsForRelativeInterpretability}). 
Another proof of the same result was given by Goris and Joosten, using the construction method 
(\cite{GorisJoosten:2011:ANewPrinciple, Joosten:1998:MasterThesis}).

The frame condition $wRxS_wyRz \Rightarrow xRz$ for \principle{M} is reflected in the following so-called labelling lemma:


\begin{lemma}[\cite{BilkvaGorisJoosten:2004:SmartLabels}, Lemma 3.7]
\label{lema-ILM}
Let $w$ and $u$ be some \extil{M}-MCS's, and let $S$ be a set of formulas.  
If $w \prec_S u$ then $w \prec_{S \cup u_\emptyset^\square} u$.
\end{lemma}

When we combine this with the main result of the previous section we get a simple, elegant and succinct completeness proof.

\begin{theorem}
The logic \extil{M} is complete w.r.t.\ \ilgen{M}-models.
\end{theorem}
\begin{proof}
Here we give the whole proof from \cite{Mikec-Vukovic-20},
    to demonstrate the interplay between labelling lemmas
    and characteristic properties. 
Proofs for other logics are similar, though usually more complex.

Given Theorem \ref{glavni}, it suffices to show that for any set $\mathcal{D}$, the \extil{M}-structure for 
$\mathcal D$ possesses the property \kgen{M}: $uS_w V \Rightarrow (\exists V' \subseteq V)( uS_w V' \ \& \ R[V']
\subseteq R[u]$. 
Let $\la W,R,\{ S_w:w\in W\},V\ra$ be the \extil{M}-structure for $\mathcal{D}.$

Let $uS_w V$ and take $V' = \{ v\in V : w \prec_{u_\emptyset^\square} v \}$. 
We claim $uS_w V'$ and $R[V']\subseteq R[u].$
Suppose $w \prec_S u$. Lemma \ref{lema-ILM} implies 
$w \prec_{S \cup u_\emptyset^\square} u.$
Since $uS_w V$, by Definition \ref{ilx-struktura}, there is $v \in V$ with $w \prec_{S \cup u_\emptyset^\square} v.$ So, $v \in V'$. Thus, $u S_w V'$. 

Now let $v \in V'$ and $z\in W$ be such that $vRz$. 
Since $v \in V'$, we know $w \prec_{u_\emptyset^\square} v$. 
Then for all $\square B \in u$ we have $\square B \in v.$ 
Since $vRz$, we have $B, \square B \in z$. 
So, $u \prec z$ and by Definition \ref{ilx-struktura} $uRz$.
\end{proof}

\subsection{The logic \extil{M_0}{}}

Modal completeness of \extil{M_0}{} w.r.t.\ ordinary Veltman semantics was proved in 
\cite{Goris-Joosten-08} by Goris and Joosten. Certain difficulties encountered in this proof were 
our main motivation for using GVS. 
We will sketch one of these difficulties and show in what way the generalised semantics 
overcomes it. 
The frame condition $wRxRyS_wuRz \Rightarrow xRz$ for \principle{M_0} is reflected in the following labelling lemma:
\begin{lemma}[\cite{BilkvaGorisJoosten:2004:SmartLabels}, Lemma 3.9]
\label{label-mn}
Let $w$, $u$ and $x$ be \extil{M_0}{}-MCS's, and $S$ an arbitrary set of formulas. 
If $w \prec_S u \prec x$ then $w \prec_{S \cup u_\emptyset^\square} x$.
\end{lemma}

To motivate our way of proving completeness (of \extil{M_0}, but also in general) w.r.t.\ GVS, let us sketch a situation for which there are clear benefits in working with GVS. We do this only now because \extil{M\textsubscript{0}} is sufficiently complex to display (some of) these benefits.
Suppose we are building models step-by-step (as in the \textit{construction method} from \cite{Goris-Joosten-08}), and worlds $w$, $u_1$, 
$u_2$ and $x$ occur in the configuration displayed in Figure \ref{fig:slika}.
Furthermore, suppose we need to produce an $S_w$-successor $v$ of $x$. 

\begin{figure}[tb]
    \centering
\tikzset{every picture/.style={line width=0.75pt}} 

\scalebox{0.8}{
\begin{tikzpicture}[x=0.75pt,y=0.75pt,yscale=-1,xscale=1]

\draw    (76.5,147.33) -- (126.84,113.45) ;
\draw [shift={(128.5,112.33)}, rotate = 506.06] [color={rgb, 255:red, 0; green, 0; blue, 0 }  ][line width=0.75]    (10.93,-3.29) .. controls (6.95,-1.4) and (3.31,-0.3) .. (0,0) .. controls (3.31,0.3) and (6.95,1.4) .. (10.93,3.29)   ;
\draw [shift={(76.5,147.33)}, rotate = 326.06] [color={rgb, 255:red, 0; green, 0; blue, 0 }  ][fill={rgb, 255:red, 0; green, 0; blue, 0 }  ][line width=0.75]      (0, 0) circle [x radius= 3.35, y radius= 3.35]   ;
\draw    (76.5,147.33) -- (126.81,179.26) ;
\draw [shift={(128.5,180.33)}, rotate = 212.4] [color={rgb, 255:red, 0; green, 0; blue, 0 }  ][line width=0.75]    (10.93,-3.29) .. controls (6.95,-1.4) and (3.31,-0.3) .. (0,0) .. controls (3.31,0.3) and (6.95,1.4) .. (10.93,3.29)   ;

\draw    (128.5,112.33) -- (177.14,147.98) ;
\draw [shift={(178.75,149.17)}, rotate = 216.24] [color={rgb, 255:red, 0; green, 0; blue, 0 }  ][line width=0.75]    (10.93,-3.29) .. controls (6.95,-1.4) and (3.31,-0.3) .. (0,0) .. controls (3.31,0.3) and (6.95,1.4) .. (10.93,3.29)   ;
\draw [shift={(128.5,112.33)}, rotate = 36.24] [color={rgb, 255:red, 0; green, 0; blue, 0 }  ][fill={rgb, 255:red, 0; green, 0; blue, 0 }  ][line width=0.75]      (0, 0) circle [x radius= 3.35, y radius= 3.35]   ;
\draw    (128.5,180.33) -- (177.05,150.22) ;
\draw [shift={(178.75,149.17)}, rotate = 508.19] [color={rgb, 255:red, 0; green, 0; blue, 0 }  ][line width=0.75]    (10.93,-3.29) .. controls (6.95,-1.4) and (3.31,-0.3) .. (0,0) .. controls (3.31,0.3) and (6.95,1.4) .. (10.93,3.29)   ;
\draw [shift={(128.5,180.33)}, rotate = 328.19] [color={rgb, 255:red, 0; green, 0; blue, 0 }  ][fill={rgb, 255:red, 0; green, 0; blue, 0 }  ][line width=0.75]      (0, 0) circle [x radius= 3.35, y radius= 3.35]   ;
\draw  [dash pattern={on 4.5pt off 4.5pt}]  (178.75,149.17) .. controls (218.35,119.47) and (217.52,176.18) .. (256.31,148.21) ;
\draw [shift={(257.5,147.33)}, rotate = 503.13] [color={rgb, 255:red, 0; green, 0; blue, 0 }  ][line width=0.75]    (10.93,-3.29) .. controls (6.95,-1.4) and (3.31,-0.3) .. (0,0) .. controls (3.31,0.3) and (6.95,1.4) .. (10.93,3.29)   ;
\draw [shift={(178.75,149.17)}, rotate = 323.13] [color={rgb, 255:red, 0; green, 0; blue, 0 }  ][fill={rgb, 255:red, 0; green, 0; blue, 0 }  ][line width=0.75]      (0, 0) circle [x radius= 3.35, y radius= 3.35]   ;
\draw    (319.5,147.33) -- (369.84,113.45) ;
\draw [shift={(371.5,112.33)}, rotate = 506.06] [color={rgb, 255:red, 0; green, 0; blue, 0 }  ][line width=0.75]    (10.93,-3.29) .. controls (6.95,-1.4) and (3.31,-0.3) .. (0,0) .. controls (3.31,0.3) and (6.95,1.4) .. (10.93,3.29)   ;
\draw [shift={(319.5,147.33)}, rotate = 326.06] [color={rgb, 255:red, 0; green, 0; blue, 0 }  ][fill={rgb, 255:red, 0; green, 0; blue, 0 }  ][line width=0.75]      (0, 0) circle [x radius= 3.35, y radius= 3.35]   ;
\draw    (319.5,147.33) -- (369.81,179.26) ;
\draw [shift={(371.5,180.33)}, rotate = 212.4] [color={rgb, 255:red, 0; green, 0; blue, 0 }  ][line width=0.75]    (10.93,-3.29) .. controls (6.95,-1.4) and (3.31,-0.3) .. (0,0) .. controls (3.31,0.3) and (6.95,1.4) .. (10.93,3.29)   ;

\draw    (371.5,112.33) -- (420.14,147.98) ;
\draw [shift={(421.75,149.17)}, rotate = 216.24] [color={rgb, 255:red, 0; green, 0; blue, 0 }  ][line width=0.75]    (10.93,-3.29) .. controls (6.95,-1.4) and (3.31,-0.3) .. (0,0) .. controls (3.31,0.3) and (6.95,1.4) .. (10.93,3.29)   ;
\draw [shift={(371.5,112.33)}, rotate = 36.24] [color={rgb, 255:red, 0; green, 0; blue, 0 }  ][fill={rgb, 255:red, 0; green, 0; blue, 0 }  ][line width=0.75]      (0, 0) circle [x radius= 3.35, y radius= 3.35]   ;
\draw    (371.5,180.33) -- (420.05,150.22) ;
\draw [shift={(421.75,149.17)}, rotate = 508.19] [color={rgb, 255:red, 0; green, 0; blue, 0 }  ][line width=0.75]    (10.93,-3.29) .. controls (6.95,-1.4) and (3.31,-0.3) .. (0,0) .. controls (3.31,0.3) and (6.95,1.4) .. (10.93,3.29)   ;
\draw [shift={(371.5,180.33)}, rotate = 328.19] [color={rgb, 255:red, 0; green, 0; blue, 0 }  ][fill={rgb, 255:red, 0; green, 0; blue, 0 }  ][line width=0.75]      (0, 0) circle [x radius= 3.35, y radius= 3.35]   ;
\draw  [dash pattern={on 4.5pt off 4.5pt}]  (421.75,149.17) .. controls (461.35,119.47) and (460.52,176.18) .. (499.31,148.21) ;
\draw [shift={(500.5,147.33)}, rotate = 503.13] [color={rgb, 255:red, 0; green, 0; blue, 0 }  ][line width=0.75]    (10.93,-3.29) .. controls (6.95,-1.4) and (3.31,-0.3) .. (0,0) .. controls (3.31,0.3) and (6.95,1.4) .. (10.93,3.29)   ;
\draw [shift={(421.75,149.17)}, rotate = 323.13] [color={rgb, 255:red, 0; green, 0; blue, 0 }  ][fill={rgb, 255:red, 0; green, 0; blue, 0 }  ][line width=0.75]      (0, 0) circle [x radius= 3.35, y radius= 3.35]   ;
\draw   (500.5,147.33) .. controls (500.51,108.56) and (531.95,77.13) .. (570.73,77.13) .. controls (609.5,77.14) and (640.93,108.58) .. (640.93,147.36) .. controls (640.92,186.14) and (609.48,217.57) .. (570.7,217.56) .. controls (531.92,217.55) and (500.49,186.11) .. (500.5,147.33) -- cycle ;
\draw   (543.92,114.34) .. controls (543.92,100.54) and (555.11,89.34) .. (568.92,89.34) .. controls (582.72,89.34) and (593.92,100.54) .. (593.92,114.34) .. controls (593.92,128.15) and (582.72,139.34) .. (568.92,139.34) .. controls (555.11,139.34) and (543.92,128.15) .. (543.92,114.34) -- cycle ;
\draw   (545.71,180.35) .. controls (545.71,166.54) and (556.91,155.35) .. (570.71,155.35) .. controls (584.52,155.35) and (595.71,166.54) .. (595.71,180.35) .. controls (595.71,194.15) and (584.52,205.35) .. (570.71,205.35) .. controls (556.91,205.35) and (545.71,194.15) .. (545.71,180.35) -- cycle ;
\draw  [dash pattern={on 4.5pt off 4.5pt}]  (371.5,112.33) .. controls (417.27,29.75) and (481.85,167.94) .. (542.99,115.15) ;
\draw [shift={(543.92,114.34)}, rotate = 498.16] [color={rgb, 255:red, 0; green, 0; blue, 0 }  ][line width=0.75]    (10.93,-3.29) .. controls (6.95,-1.4) and (3.31,-0.3) .. (0,0) .. controls (3.31,0.3) and (6.95,1.4) .. (10.93,3.29)   ;
\draw [shift={(371.5,112.33)}, rotate = 299] [color={rgb, 255:red, 0; green, 0; blue, 0 }  ][fill={rgb, 255:red, 0; green, 0; blue, 0 }  ][line width=0.75]      (0, 0) circle [x radius= 3.35, y radius= 3.35]   ;
\draw  [dash pattern={on 4.5pt off 4.5pt}]  (371.5,180.33) .. controls (437.17,253.96) and (463.24,139.49) .. (544.48,179.73) ;
\draw [shift={(545.71,180.35)}, rotate = 207.07] [color={rgb, 255:red, 0; green, 0; blue, 0 }  ][line width=0.75]    (10.93,-3.29) .. controls (6.95,-1.4) and (3.31,-0.3) .. (0,0) .. controls (3.31,0.3) and (6.95,1.4) .. (10.93,3.29)   ;
\draw [shift={(371.5,180.33)}, rotate = 48.27] [color={rgb, 255:red, 0; green, 0; blue, 0 }  ][fill={rgb, 255:red, 0; green, 0; blue, 0 }  ][line width=0.75]      (0, 0) circle [x radius= 3.35, y radius= 3.35]   ;

\draw (257.5,147.33) node   {$\CIRCLE $};
\draw (76,163) node   {$w$};
\draw (113,102) node   {$u_{1}$};
\draw (113.5,187) node   {$u_{2}$};
\draw (125,132.5) node   {$\square B_{1}$};
\draw (125,161.5) node   {$\square B_{2}$};
\draw (368,132.5) node   {$\square B_{1}$};
\draw (368,161.5) node   {$\square B_{2}$};
\draw (261,131.5) node   {$\square B_{1} ,\ \square B_{2}$};
\draw (610,126.5) node   {$\square B_{1}$};
\draw (610,167.5) node   {$\square B_{2}$};
\draw (179,164) node   {$x$};
\draw (319,164) node   {$w$};
\draw (356,103) node   {$u_{1}$};
\draw (356.5,186) node   {$u_{2}$};
\draw (422,165) node   {$x$};
\draw (257,165) node   {$v$};
\draw (510,146) node   {$V$};
\draw (557,180) node   {$V_{2}$};
\draw (556,114) node   {$V_{1}$};
\end{tikzpicture}
}

\caption{Left: extending an ordinary Veltman model. Right: extending a generalised Veltman model. Straight lines represent $R$-transitions, while curved lines represent $S_w$-transitions. Full lines represent the starting configuration, and dashed lines represent the transitions that are to be added.
This figure is also taken from \cite{Mikec-Vukovic-20}.}
\label{fig:slika}
\end{figure}
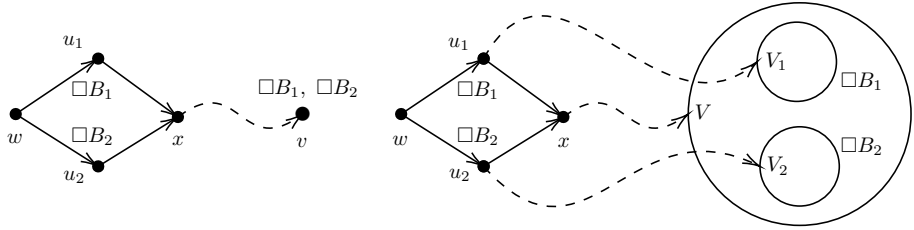

With the ordinary semantics, we need to ensure that for our $S_w$-successor $v$, for each 
$\square B_1 \in u_1$ and $\square B_2 \in u_2$, we have $\square B_1, \square B_2 \in v$. 
It is not obvious that such a construction is possible. 
In case of \extil{M_0}, it was successfully solved in \cite{Goris-Joosten-08} by preserving the invariant 
that sets of boxed formulas in $u_i$ are linearly ordered. 
This way, finite (quasi-)models can always be extended by only looking at the last $u_i$. 

With GVS, we need to produce a whole set of worlds $V$, but the requirements from the frame condition $wRuRxS_wV\ \Rightarrow \ (\exists V'\subseteq V) 
(uS_wV' \ \& \ R[V']\subseteq R[u]))$ on 
each particular world are less demanding. 
For each $u_i$, there has to be a corresponding $V_i \subseteq V$ with $\square B_i$ contained 
(true) in every world of $V_i$. 
Lemma \ref{label-mn} gives a recipe for producing such worlds.

\begin{theorem}
The logic \extil{M_0}{} is complete w.r.t.\ \ilgen{M\textsubscript{0}}-models.
\label{th-ilmn-complete}
\end{theorem}
\begin{proof}
Omitted. See \cite{Mikec-Vukovic-20} for details.
\end{proof}

\subsection{The logics \extil{P}, \extil{P_0} and \extil{R} }

The logics \extil{P}, \extil{P_0} and \extil{R} can be proven to be complete with respect to their classes of frames in a similar way (see \cite{Mikec-Vukovic-20} for details).

We recall that the interpretability logic \extil{P_0} is incomplete w.r.t.\ Veltman models (\cite{GorisJoosten:2011:ANewPrinciple}). 
Since \extil{P_0} is 
complete w.r.t.\  GVS, this is the first example of 
an interpretability logic complete w.r.t.\ GVS, but incomplete w.r.t.\ 
ordinary semantics.


\subsection{The logics \extil{W} and \extil{W^*}}\label{section:LogicsWandStar}

To prove that \extil{W} is complete, one could try to find a sufficiently strong ``labelling lemma'' and use 
Definition \ref{ilx-struktura} (\extil{X}-structure). One candidate might be the following condition:
\[
    w \prec_S u \ \Rightarrow \ (\exists G \in \mathcal{D}) \ \Big( w \prec_{S \cup \{ \square \neg G \} } u 
    \ \&\  G \in u \Big),
\]
where $\mathcal{D}$ is finite, closed under subformulas and such that each $w \in W$ contains $A_w$ and 
$\Box \neg A_w$ for some $A_w \in \mathcal D$. 
If there is such a condition, 
    it would greatly simplify proofs of completeness for extensions of \extil{W}.
Unfortunately, at the moment we do not know if such a condition can be formulated and proved.

Another approach is to use a 
modified version of Definition \ref{ilx-struktura} to work with  \extil{W} and its extensions. This way we won't 
require a labelling lemma, but we lose generality in the following sense. To prove the completeness of \extil{XW}, 
for some $X$, it no longer suffices to simply show that the structure defined in Definition \ref{ilx-struktura} 
has the required characteristic property (when each world is an \extil{X}-MCS). Instead, the characteristic 
property of \extil{X} has to be shown to hold on the modified structure. So, to improve compatibility with proofs 
based on Definition \ref{ilx-struktura}, we should prove the completeness of $\extil{W}$ with a definition as similar 
 to Definition \ref{ilx-struktura} as possible. That is what we do in the remainder of this section. 
This approach turns out to be good enough for \extil{W^*} (\extil{WM_0}). We didn't succeed in using 
it to prove the completeness of \extil{WR}. However, to the best of our knowledge, \extil{WR} might not be complete 
at all.

We have already mentioned the \textsf{Not-W} frame condition in Definition \ref{defn:not_w} that characterises when \principle{W} fails on GVS.
The positive condition \kgen{W} from \cite{Mikec-Perkov-Vukovic-17} is given by:
\[
\kgen{W} \ : = \    uS_w V \ \Rightarrow \  (\exists V' \subseteq V)\ \big(\, u S_w V' \ \& \ R[V'] \cap 
    S_w^{-1}[V] = 
    \emptyset \, \big). 
\]
We will use (this formulation of) \kgen W in what follows. We note here that the \kgen W condition can be formulated in a more informative way. 
Whenever there are $w$, $u$ and $V$ such that $(w, u, V)$ is a counterexample to \kgen W,
    there is $U \subseteq V$ such that:
\begin{enumerate}
    \item[(i)]  $(w, u, U)$ is a counterexample to \kgen W;
    \item[(ii)] $R[U] \cap U = \emptyset$;
    \item[(iii)] there are sets $U_0$ and $\overline{U}$ such that: 
    \begin{itemize}
    \item
    $U = U_0 \cup \overline{U}$ and $\overline{U} \neq \emptyset$;
        \item $U_0 = \{ v \in U : R[v] \cap S_w^{-1}[U] = \emptyset \}$;
        \item $\overline{U} = \{ v \in U \setminus U_0 : \forall V'\ \big(\, \exists z\ vRzS_w V'\subseteq U \ \Rightarrow \ 
                                                      V' \cap \overline{U} \neq \emptyset \, \big) \}$.
    \end{itemize}
\end{enumerate}
This new formulation tells us that we can pick a set $U$ and a quasi-partition $\{ U_0, \overline{U} \}$ of $U$
    such that points in $U_0$ cannot ``return'' to $U$, 
    while the points in $\overline{U}$ can ``return'' to $U$, and have an additional property
    that every such ``return'' intersects (not only $U$ but also) $\overline{U}$.
The proof that such $U$ can always be found will be available in the third author's PhD thesis (\cite{luka-phd}).

In the proof of completeness of logic \extil{W} we will use the following two lemmas. In what follows, \extil{WX} 
denotes an arbitrary extension of \extil{W}.

\begin{lemma}[\cite{BilkvaGorisJoosten:2004:SmartLabels}, Lemma 3.12]
\label{problemi-W}
Let $w$ be an \extil{WX}-MCS, and $B$ and $C$  formulas such that $\neg (B\rhd C)\in w.$ 
Then there is an \extil{WX}-MCS $u$ such that $w{\prec_{\{ \square\neg B,\neg C\}}}u$ and  $B\in u.$
\end{lemma}

\begin{lemma}[\cite{BilkvaGorisJoosten:2004:SmartLabels}, Lemma 3.13]
\label{nedostaci-W}
Let $w$ and $u$ be some \extil{WX}-MCS, $B$ and $C$ some formulas, and $S$ a set of formulas such that 
$B\rhd C\in w,$ $w\prec_S u$  and $B\in u.$
Then there is an \extil{WX}-MCS $v$ such that $w\prec_{S\cup \{ \square\neg B\} } v$ and 
$C,\square\neg C\in v.$
\end{lemma}

Given a binary relation $R$, let $\dot{R}[x] =R[x]\cup \{ x\}.$ 
If the set $\dot{R}[x]$ contains maximal consistent sets
    (which it usually does in this section),
    then $\bigcup\dot{R}[x]$ is a set of formulas.
If satisfaction coincides with formulas contained, 
    then it is useful to think of $\bigcup\dot{R}[x]$
    as the set of formulas $B$ such that either
    $B$ or $\Diamond B$ is satisfied in $x$
    (however, one has to be careful with such an interpretation,
    since we do not claim a truth lemma to hold for all formulas).

\begin{definition}
\label{def-ilw-struktura}
Let \textsf{X} be \textsf{W} or \textsf{W$^*$}. We say that $\mathfrak{M} = \la W, R, \{S_w : w \in W\}, V\ra$ 
is \emph{the \extil{X}-structure for a set of formulas $\mathcal{D}$} if:
        \[
            \begin{array}{rll}
                   W &:=& \{ w : w \text{ is an \extil{X}-MCS and for some } G \in \mathcal{D}, G 
                   \wedge \square \neg G \in w \};\\
                    wRu &:\Leftrightarrow & w \prec u;\\  
                    uS_w V &:\Leftrightarrow & wRu \mbox{\ and, } V \subseteq R[w]\mbox{ and, one of the following holds:}\\
                   &  & (a) \  V \cap \dot{R}[u] \neq \emptyset;\\
                    & & (b) \ (\forall S) \left (w \prec_S u \Rightarrow (\exists v \in V) \left (\exists G \in 
                \mathcal{D} \cap \bigcup\dot{R}[u] \right ) \ w \prec_{S\cup \{\square\neg G\} } v \right );\\
                    w\in V(p) &: \Leftrightarrow & p\in w.
            \end{array}
        \]
\end{definition}
With this definition, we can now prove a truth lemma.
\begin{lemma}
Let \textsf{X} be \textsf{W} or \textsf{W*}. The \extil{X}-structure $\mathfrak{M}$ for $\mathcal{D}$ is a generalised 
Veltman model. Furthermore, the 
following holds:
\[ 
    \mathfrak{M},w\Vdash G \ \mbox{ if and only if }\ G\in w,
\]
for each $G\in\mathcal{D}$ and $w\in W.$ 
\label{lemma-main-w}
\end{lemma}
\begin{proof}
The proof of this claim is lengthy (mostly due to the way quasi-transitivity is defined in Definition \ref{def-ilw-struktura}). However, the proof is straightforward. 
See our comment in the proof of Lemma \ref{lemma-main}.
For details please refer to \cite{Mikec-Vukovic-20}.
\end{proof}
This lemma brings us one step away from a completeness proof. We first introduce the following notation: Let $B$ be a formula, and $w$ a world in a generalised Veltman model.
We write $[B]_w$ for $\{ u : wRu$ and $u\Vdash B\}$.
\begin{theorem}
The logic \extil{W} is complete w.r.t.\ \ilgen{W}-models.
\label{tm-ilw}
\end{theorem}
\begin{proof}
Sketch.
In the light of Lemma \ref{lemma-main-w}, it suffices to show that the \extil{W}-structure $\mathfrak{M}$ for $\mathcal{D}$ possesses the property 
\kgen{W}. Recall the characteristic property \kgen{W}: 
\[
  uS_w V \ \Rightarrow \  (\exists V' \subseteq V)\ \big(\, u S_w V' \ \& \ R[V'] \cap 
    S_w^{-1}[V] = 
    \emptyset \, \big). 
\]
Suppose for a contradiction that there are $w$, $u$ and $V$ such that:
\begin{equation}
    \label{newgen}
     u S_w V \ \& \ (\forall V' \subseteq V)(u S_w V' \Rightarrow R[V'] \cap 
     S_w^{-1}[V] \neq \emptyset ).
\end{equation}
Let $\mathcal{V}$ denote the collection of all such sets $V$ (keeping $w$ and $u$ fixed). 

Let $n = 2^{|\mathcal{D}|}$. Fix any enumeration $\mathcal{D}_0, \dots, \mathcal{D}_{n - 1}$  of $\mathcal{P}(\mathcal{D})$ that satisfies $\mathcal{D}_0=\emptyset.$
We define a new relation $S_w^i$ for each $0 \leq i < n$
as follows:
   		\[
	  		y S_w^i U :\iff yS_w U, \ \mathcal{D}_i \subseteq \bigcup\dot{R}[y], \ U \subseteq 
	  		\left[\bigvee_{G\in\mathcal{D}_i} \square \neg G \right]_w.
		\]
It can be shown that whenever $yS_w U$, we also have:
\begin{equation}
	\label{sv1}
	(\exists U' \subseteq U)(\exists i < n) \ y S_w^{i} U'.
\end{equation}

Let $m < n$ be maximal such that there are $U \in \mathcal{V}$ and $U' \subseteq U$ with 
the following properties:
\begin{enumerate}
		\item[(i)] $(\forall x \in U)[ (\exists y \in R[x]) (\exists Z \subseteq U) 
		(\exists i \leq m)\ yS_w^{i} Z \Rightarrow x \notin U' ]$;
		\item[(ii)] $(\forall x \in W)(x S_w U \Rightarrow  xS_w U')$.
\end{enumerate}
		
Since $\mathcal{D}_0=\emptyset,$ we have
$[\bigvee_{G\in\mathcal{D}_0} \square\neg G]_w=[\bot]_w=\emptyset.$
So there are no $Z\subseteq [\bigvee_{G\in\mathcal{D}_0}\square\neg G]_w$ such that $yS_w Z$
for some $y\in W$. So, if we take $m = 0$ and $U' = U$ for any $U \in \mathcal{V}$, (i) and (ii) are 
trivially satisfied. 

Since $n$ is finite and conditions (i) and (ii) are satisfied for at least one value $m$, there must be a 
maximal $m < n$ with the required properties.

Finally, it can be shown that, contrary to the assumption of the maximality of $m$, $m+1$ also satisfies properties (i) and (ii). 
We omit details since the full proof is somewhat cumbersome.
For details please refer to \cite{Mikec-Vukovic-20}.
\end{proof}

Goris and Joosten proved in \cite{Goris-Joosten-08} the completeness of \principle{W}$^*$ (recall that this is equivalent to \extil{WM_0}) w.r.t.\ ordinary  Veltman semantics. This theorem has its analogue in GVS.

\begin{theorem}
\label{ilwst-potpunost}
The logic \extil{W^*} is complete w.r.t.\ \ilgen{W} $^*$-models.
\end{theorem}
\begin{proof}
From Lemma \ref{lemma-main-w}, it suffices to prove that the \extil{W^*}-structure for $\mathcal{D}$
possesses the properties \kgen{W} and \kgen{M$_0$}, for each appropriate    $\mathcal{D}.$
So, let $\mathfrak{M}=\la W,R,\{ S_w: w\in W\},V\ra$ 
be the \extil{W^*}-structure for $\mathcal D$.
Theorem \ref{tm-ilw} shows that the model $\mathfrak{M}$ possesses the property \kgen{W}.
It remains to show that it possesses the property \kgen{M$_0$}.

The remainder of this proof is very similar to the proof of Theorem \ref{th-ilmn-complete}. Please refer to \cite{Mikec-Vukovic-20} for details.
\end{proof}

In \cite{Mikec-Perkov-Vukovic-17} it is shown that \extil{W^*} possesses the finite model property w.r.t.\ generalised 
Veltman models. To show decidability, (stronger) completeness w.r.t.\ ordinary Veltman models was used in \cite{Mikec-Perkov-Vukovic-17}. However, we observe that 
Theorem \ref{ilwst-potpunost} above suffices for the mere purpose of decidability.

\subsection{The logic \extil{WR}}

In previous subsections we saw that the completeness of \extil{R} can be proven
    using \extil{R}-structures, and that the completeness of \extil{W}
    can be proven using \extil{W}-structures.
These two types of structures are defined differently.
However, we saw that \extil{W^*}-structures have the same general form 
    as \extil{W}-structures.
So, one may hope to prove completeness of \extil{WR} with the help
    of a structure defined similarly to \extil{W}-structures.
    
Unfortunately, it seems that \extil{WR}-structures, if by an
    \extil{WR}-structure we mean an \extil{W}-structure
    with the notion of \extil{W}-consistency replaced with that of
    \extil{WR}-consistency,
    does not posses the characteristic property \kgen{R}.
In \cite{goris2020assuring} we call the type of a problem that emerges
    here ``the label iteration problem''.
In the same paper we demonstrate how to overcome this problem
    for a simpler logic.
With \extil{WR} we have some progress, 
    but are not yet sure if we can really solve it.

\subsection{Logics below \il}

Just as this chapter was being prepared, a preprint written by Kurahashi and Okawa
appeared (\cite{kurahashi2020modal}), using GVS to prove completeness and decidability of certain subsystems of \extil{}.
The authors define a new logic, \extil{^-},  similarly to \extil{}, but without axiom (schema)s \textsf{J1}, \textsf{J2}, \textsf{J4} and \textsf{J5}. 
However, they add new rules to the system: \begin{center}
    \textsf{R1}: if $\vdash A \to B$ then $\vdash C \rhd A \to C \rhd B$;\\
    \textsf{R2}: if $\vdash A \to B$ then $\vdash B \rhd C \to A \rhd C$.
\end{center} 
These new rules can be seen as approximating \textsf{J1} and \textsf{J2}. The authors also require $\Box A \leftrightarrow \neg A \rhd \bot$ to hold by definition.

The paper proceeds to study twenty logics between \extil{^-} and \extil{}.
Twelve of these are proven to be complete with respect to a version of ordinary Veltman semantics.
The remaining eight are incomplete w.r.t.\ such semantics, and complete with respect to a version of GVS. 
Their style of proof is similar to \cite{JonghVeltman:1990:ProvabilityLogicsForRelativeInterpretability} and \cite{Mikec-Vukovic-20}. The authors define their whole models all at once (in a non-iterative construction), and the general structure of the definitions of relations $S_w$ is, roughly, ``$uS_wV$ if whenever there is a label $\Sigma$ between $w$ and $u$, there should be a corresponding world $v \in V$ with the same label $\Sigma$ between $w$ and $v$''.

The results from  \cite{kurahashi2020modal} on subsystems of \il lend support to the conviction that Generalised semantics is robust and widely applicable. See also Remark 4 in Section 1.3 of \cite{LitakVisserInternal:2020:ForDick}.

\section{Bisimulations and filtrations}

In this section we introduce and establish basic properties of bisimulations between generalised 
Veltman models. Next, we use bisimilarity to define equivalence classes when we employ 
the method of filtrations, which in turn we use to prove finite model property and decidability of 
various logics.

\subsection{Bisimulations}

Visser \cite{Visser:1990:InterpretabilityLogic} defined the notion of a bisimulation between 
Veltman models. Vrgo\v c and Vukovi\'c \cite{Vrgoc-Vukovic} extended this definition to 
generalised Veltman models.

\begin{definition}
A bisimulation between generalised Veltman models\\
$\mathfrak{M}=\la W,R,\{S_w:w\in W\},V\ra$ 
and $\mathfrak{M}'=\la W',R',\{S'_{w'}:w'\in W'\},V'\ra$ is a non-empty
relation $Z\subseteq W\times W'$ such that:
\begin{quote}
\begin{itemize}
\item[(at)] if $wZw'$, then $\mathfrak{M},w\Vdash p$ if and only if $\mathfrak{M}',w'\Vdash' p$, 
for all propositional variables $p;$

\item[(forth)] if $wZw'$ and $wRu$, then there is $u'\in W'$ such that $w'R'u'$, $uZu'$ and for 
all $V'\subseteq W'$ such that $u'S'_{w'}V'$ there is $V\subseteq W$ such that $uS_wV$ and for
all $v\in V$ there is $v'\in V'$ with $vZv';$ 

\item[(back)] if $wZw'$ and $w'R'u'$, then there is $u\in W$ such that $wRu$, $uZu'$ and for all 
$V\subseteq W$ such that $uS_wV$ there is $V'\subseteq W'$ such that $u'S'_{w'}V'$ and for
all $v'\in V'$ there is $v\in V$ with $vZv'.$ 
\end{itemize}
\end{quote}

We say that $w\in W$ and $w'\in W'$ are \textit{bisimilar} if there is a bisimulation 
$Z\subseteq W\times W'$ such that $wZw'.$
\end{definition}

The following lemma is proved by Vrgo\v{c} and Vukovi\'{c} in \cite{Vrgoc-Vukovic}.

\begin{lemma}
Let $\mathfrak M$, $\mathfrak{M}'$ and $\mathfrak{M}''$ be generalised Veltman models.

\begin{itemize}
\item[a)] If $w\in W$ and $w'\in W'$ are bisimilar, then they are modally equivalent, i.e.\ they 
satisfy the same formulas in the language of \extil{}.

\item[b)] The identity $\{(w,w):w\in W\}\subseteq W\times W$ is a bisimulation.

\item[c)] The inverse of a bisimulation between $\mathfrak{M}$ and $\mathfrak{M}'$ is a 
bisimulation between $\mathfrak{M}'$ and $\mathfrak{M}.$

\item[d)] The composition of bisimulations $Z\subseteq W\times W'$ and 
$Z'\subseteq W'\times W''$ is a bisimulation between $\mathfrak{M}$ and $\mathfrak{M}''.$

\item[e)] The union of a family of bisimulations between $\mathfrak{M}$ and $\mathfrak{M}'$ is 
also a bisimulation between $\mathfrak{M}$ and $\mathfrak{M}'$. 
Thus there exists the largest bisimulation between models $\mathfrak{M}$ and $\mathfrak{M}'.$
\end{itemize}
\label{lema-bismulacije}
\end{lemma}
\noindent

The previous lemma shows that this notion of bisimulation has certain desired properties. 
A property that significantly contributes to whether the notion of bisimulation can be considered 
well-behaved is the Hennessy--Milner property. We say that a generalised Veltman model 
$\mathfrak{M}=\la W,R;\{ S_w :w\in W\},V\ra$ is \textit{image finite} if the set 
$R[w]$ is finite, for all $w\in W$.
The following theorem is proved in 
\cite{Vrgoc-Vukovic}.

\begin{theorem}[Hennessy-Milner\footnote{The theorem for unary modal logic was already known to and published by van Benthem \cite{van1984correspondence}.} property] 
Let $\mathfrak{M}=\la W,R,\{S_w:w\in W\},V\ra$ and 
$\mathfrak{M}'=\la W',R',\{ S_w' : w\in W'\},V' \ra$ be two image finite generalised Veltman 
models. If $w\in W$ and $w'\in W'$ are modally equivalent, then there exists a bisimulation $Z$ 
such that $w Z w'$.
\end{theorem}

In \cite{Vrgoc-Vukovic} several other notions of a bisimulation (V-bisimulation, strong and 
global bisimulation) are considered and the connections between them are explored. 

In \cite{Verbrugge} and \cite{Vukovic08}, connections between Veltman semantics and generalised 
Veltman semantics are considered. 
In \cite{Vukovic08} it is shown that for a restricted class of generalised Veltman models 
$\mathfrak{M}$ (the so-called \textit{complete image finite} models) there exists an ordinary 
Veltman model $\mathfrak{M}'$ that is bisimilar to $\mathfrak{M}$.

In the section concerning completeness, we commented on why in general there cannot 
exist an ordinary \extil{P_0}-model that is bisimilar to a given \ilgen{P$_0$}-model.
That example shows that for at least some cases ordinary Veltman semantics is not expressive 
enough to capture the behaviour of interpretability logics.

In \cite{Perkov-Vukovic-16} the notion of $n$-bisimulation is defined similar to various existing notions for other modal logics. Then, $n$-bisimulations are used in the proof of the finite model property of various systems w.r.t.\ GVS.

\begin{definition}
An $n$-\textit{bisimulation} between generalised Veltman models 
$\mathfrak{M}=\la W,R,\{S_w:w\in W\},V\ra$ and 
$\mathfrak{M}'=\la W',R',\{S'_{w'}:w'\in W'\},V'\ra$ is a decreasing
sequence of relations $Z_n\subseteq Z_{n-1}\subseteq\dots\subseteq Z_1\subseteq Z_0\subseteq 
W\times W'$ such that:
\begin{quote}
\begin{itemize}
\item[(at)] if $wZ_0w'$ then $\mathfrak{M},w\Vdash p$ if and only if 
$\mathfrak{M}',w'\Vdash' p$, for all propositional variables $p;$

\item[(forth)] if $0<i\leqslant n$, $wZ_iw'$ and $wRu$, then there exists $u'\in R'[w']$ 
such that $uZ_{i-1}u'$ and for all $V'\in S'_{w'}[u']$ there is
$V\in S_w[v]$ such that for all $v\in V$ there is $v'\in V'$ with $vZ_{i-1}v';$

\item[(back)] if $0<i\leqslant n$, $wZ_iw'$ and $w'R'u'$, then there exists $u\in R[w]$ 
such that $uZ_{i-1}u'$ and for all $V\in S_w[u]$ there is $V'\in S'_{w'}[u']$ such that for all 
$v'\in V'$ there is $v\in V$ with  $vZ_{i-1}v'.$
\end{itemize}
\end{quote}
We say that $w\in W$ and $w'\in W'$ are $n$-\textit{bisimilar} if there is an $n$-bisimulation 
between $\mathfrak{M}$ and $\mathfrak{M}'$ such that $wZ_nw'$. 
\end{definition}

We say that $w$ and $w'$ are $n$-\textit{modally equivalent} and we write $w\equiv_nw'$ if $w$ 
and $w'$ satisfy exactly the same formulas of the modal depth (that is, the maximal number of 
nested modalities) up to $n$. The following lemma is proved in \cite{Perkov-Vukovic-16}.

\begin{lemma}
Let $\mathfrak{M}=\la W,R,\{ S_w w\in W\}, V\ra$ and 
$\mathfrak{M}'=\la W',R', \{ S_w' : w\in W'\}, V'\ra$ be 
generalised Veltman models. Let $w\in W$ and $w'\in W'$. We have:
\begin{enumerate}
\item if $w$ and $w'$ are $n$-bisimilar, then $w$ and $w'$ are $n$-modally equivalent;

\item if there are only finitely many propositional variables, then the converse also holds: 
if $w$ and $w'$ are $n$-modally equivalent, then $w$ and $w'$ are $n$-bisimilar.
\end{enumerate}
\label{lm:nbis}
\end{lemma}

\subsection{Filtrations and the finite model property}
The filtration method is often used to prove that a modal logic possesses the finite model 
property.
Perkov and Vukovi\'c \cite{Perkov-Vukovic-16}  applied this technique to GVS. 
Here bisimilarity is used to refine models in order to preserve the structural properties of 
generalised Veltman models.

Filtration usually employs partitions whose clusters contain logically equivalent worlds. The 
equivalence need not be with respect to all formulas; usually a finite set of formulas closed 
under taking subformulas suffices. 
In our case, some additional properties of this set of formulas are required. 
Let $A$ be a formula. If $A$ is not a negation, then ${\sim}A$ denotes $\neg A$, and otherwise,
if $A$ is $\neg B$, then ${\sim}A$ is $B$. 
It is convenient to take $\rhd$ to be the only modality in our language and to define $\Box$ 
and $\Diamond$ as abbreviations: $\Diamond A$ as $\neg(A\rhd\bot)$ and
$\Box A$ as ${\sim}\Diamond{\sim}A$, i.e.\ ${\sim}A\rhd\bot$.
We will give the definition of adequate sets used in \cite{Mikec-Vukovic-20}.
It is an extended version of the definition used in \cite{Perkov-Vukovic-16}, where filtrations 
of generalised Veltman models were originally introduced. 
The extended definition turns out to be important for some logics.

\begin{definition}
Let $\mathcal{D}$ be a finite set of formulas that is closed under taking 
subformulas and single negations $\sim$, and $\top\in \mathcal{D}.$
We say that a set of formulas $\Gamma_\mathcal{D}$ is \emph{an adequate set} (w.r.t.\  $\mathcal{D}$) if 
it satisfies the following conditions:
\begin{itemize}
\item[a)] $\Gamma_\mathcal{D}$ is closed under taking subformulas;
\item[b)] if $A\in\Gamma_\mathcal{D}$ then ${\sim}A\in\Gamma_\mathcal{D};$
\item[c)] $\bot\rhd\bot\in\Gamma_\mathcal{D};$
\item[d)] $A\rhd B\in\Gamma_\mathcal{D}$ if $A$ is an antecedent or succedent of some $\rhd$-formula 
             in $\Gamma_\mathcal{D},$ and so is $B;$
\item[e)] if $A\in\mathcal{D}$ then $\square\neg A\in \Gamma_\mathcal{D}.$    
\end{itemize}
\end{definition}

Since the set of formulas $\mathcal{D}$ is finite, $\Gamma_\mathcal{D}$ is finite too.
Now, let $\mathfrak{M}=\la W,R,\{S_w:w\in W\},V\ra$ be a generalised Veltman model and let $\Gamma_{\mathcal{D}}$ 
be an adequate set of formulas w.r.t.\ some appropriate set $\mathcal D$. 
For nodes $w,u\in W$, we write $w\equiv_{\Gamma_{\mathcal{D}}} u$ if for all $A\in\Gamma_{\mathcal{D}}$ we 
have $\mathfrak{M},w\Vdash A$ if and only if $\mathfrak{M},u\Vdash A.$

Let $\sim\;\subseteq\;\equiv_{\Gamma_{\mathcal{D}}}$ be an equivalence relation on the set $W$. Denote the 
$\sim$-equivalence class of $w\in W$ by $[w]$, and $\widetilde{V}=\{[w]:w\in  V\}$ for any $V\subseteq W.$

A \textit{filtration} of a model $\mathfrak{M}$ through $\Gamma_{\mathcal{D}},\sim$ is any generalised Veltman model
$\widetilde{\mathfrak{M}}=\la\widetilde{W},\widetilde{R},\{\widetilde{S}_{[w]}:[w]\in \widetilde{W}\},V' \ra$ such 
that for all $w\in W$ and $A\in\Gamma_{\mathcal{D}}$ we have $\mathfrak{M},w\Vdash A$ if
and only if $\widetilde{\mathfrak{M}},[w]\Vdash' A$. 
Fact e) of Lemma \ref{lema-bismulacije} implies that the largest bisimulation 
$\sim_{\mathfrak{M}}$ of model $\mathfrak{M}$ exists. 
Facts b), c) and d) of the same lemma imply that $\sim_{\mathfrak{M}}$ is an equivalence relation, while a) 
implies $\sim_{\mathfrak{M}}\;\subseteq\;\equiv_{\Gamma_{\mathcal{D}}}.$

\begin{lemma} [\cite{Perkov-Vukovic-16}, Lemma 2.3 and Theorem 2.4] Let $\mathfrak{M}=\la W,R,\{S_w:w\in W\},V\ra$ 
be a generalised Veltman model, $\Gamma_{\mathcal{D}}$ an adequate set of formulas, and $\sim_{\mathfrak{M}}$ 
the largest bisimulation of model $\mathfrak{M}.$ Let us define:

\begin{itemize}
\item[a)] $\widetilde{R}=\{([w],[u]):wRu $ and there is $\Box A\in\Gamma_{\mathcal{D}}$ such that 
$\mathfrak{M},w\not\Vdash\Box A$ and $\mathfrak{M},u\Vdash\Box A\};$

\item[b)] $[u]\widetilde{S}_{[w]}\widetilde{V}$ if and only if $[w]\widetilde{R}[u]$, 
$\widetilde{V}\subseteq \widetilde{R}\big[[w]\big]$, 
and for all $w'\in[w]$ and $u'\in[u]$ such that $w'Ru'$ we have 
$u'S_{w'}V'$ for some $V'$ such that $\widetilde{V'}\subseteq \widetilde{V};$

\item[c)] for all propositional variables $p\in\Gamma_{\mathcal{D}}$ put $[w]\in V'(p)$ if and only if 
$w\in V(p),$ and interpret propositional variables
          $q\notin\Gamma_{\mathcal{D}}$ arbitrarily (e.g.\ put $[w]\not\in   V(q)$ for all $[w]\in\widetilde{W}).$
\end{itemize}
Then $\widetilde{\mathfrak{M}}=\la \widetilde{W},\widetilde{R},\{\widetilde{S}_{[w]}:[w]\in \widetilde{W}\},V'\ra$ 
is a filtration of the model $\mathfrak{M}$ through $\Gamma_{\mathcal{D}},\sim_{\mathfrak{M}}$.
\label{lm3}
\end{lemma}

We can use the construction above to prove the finite model property of many logics.
Let us briefly sketch the proof. We start by fixing a formula $A$ satisfied in some model $\mathfrak M$ and a 
finite adequate set $\Gamma_{\mathcal{D}}$ such that $A\in\Gamma_{\mathcal{D}}$.
For the purposes of this proof, we assume the language contains only the propositional variables that are 
contained in $A$. Let $\widetilde{\mathfrak{M}}$ be a filtration of this model as described above.
We first prove that the length of $\widetilde{R}$-chains is bounded by the number of occurrences of boxed 
formulas in $\Gamma_{\mathcal{D}}$. This implies that bisimilarity can be simplified to $n$-bisimilarity, 
for a sufficiently large $n$. Lemma \ref{lm:nbis} implies that clusters are $n$-bisimilar if and only if they 
are modally $n$-equivalent. Thus there can be only finitely many classes with respect to $n$-equivalence, 
implying our model is finite.

\begin{theorem}[\cite{Perkov-Vukovic-16}]
The logic \extil{} has the finite model property with respect to generalised Veltman models.
\label{tm3}
\end{theorem}

For \extil{}, \extil{M}, \extil{P} and \extil{W}, the original completeness proofs were proofs of 
completeness w.r.t.\ appropriate finite models  
\cite{JonghVeltman:1990:ProvabilityLogicsForRelativeInterpretability}, 
\cite{JonghVeltman:1999:ILW}. 
For these logics, the FMP w.r.t.\ the ordinary semantics and decidability are immediate (and 
completeness and the 
FMP w.r.t.\ GVS are easily shown to follow from these results).
For more complex logics, not much is known about the FMP w.r.t.\ the ordinary semantics. 

To prove that a specific extension has the FMP, it remains to show that  filtration preserves its 
characteristic property.  This approach was successfully used to prove the FMP of 
\extil{M_0}, \extil{W^*}, 
\extil{P_0} and \extil{R} w.r.t.\ GVS \cite{Perkov-Vukovic-16}, 
\cite{Mikec-Perkov-Vukovic-17}, \cite{Mikec-Vukovic-20}.

Since we have the finite model property, and ``finite'' can be taken to mean finite in every sense 
(i.e.\ there is a finite code, obtainable in a straightforward manner, for every such model),
we also have decidability. This follows by the standard argument: enumerate all the proofs
(which is possible since all the logics \extil{X} in question are recursively enumerable)
and all the (codes of) finite models simultaneously. Sooner or later, we either find a proof 
of $A$, or, because of the completeness and the FMP, a model of $\neg A$.
\begin{corollary}
The logics \extil{M_0}, \extil{W^*}, \extil{P_0} and \extil{R}
        are decidable.
\end{corollary}


\section{Hierarchies and frame conditions}\label{section:Hierarchies}

In this final section we will present frame conditions of a new series of principles in \ilal as 
presented in \cite{GorisJoosten:2020:TwoSeries}. We shall recall the frame-conditions computed for 
this series with respect to regular Veltman semantics. Next, we shall present the respective frame 
conditions for the GVS. These novel results are moreover formalised in the proof assistant Agda 
and will shortly be available in the master thesis of  Mas Rovira (\cite{MasRovira:2020:MastersThesis}) written under the direction of 
Joosten and Mikec.

\subsection{A broad series of principles}

Let us present here one of two series from \cite{GorisJoosten:2020:TwoSeries}. By lack of a better name the series was called \emph{the broad series}. To present this series we first define a series of auxiliary formulas. For any $n\geq 1$ we define the schemata $\uu_n$ as follows.
\begin{align*}
\uu_1 &:= \Diamond\neg(D_1\rhd \neg C),\\
\uu_{n+2} &:= \Diamond((D_{n+1}\rhd D_{n+2})\wedge\uu_{n+1}).
\end{align*}
Now, for $n\geq 0$ we define the schemata for the broad series $\principle R^n$ as follows.
\begin{align*}
\principle{R}^0 &:= A\rhd B\rightarrow\neg(A\rhd \neg C)\rhd B\wedge\Box C,\\
\principle{R}^{n+1} &: = A\rhd B\rightarrow\uu_{n+1}\wedge(D_{n+1}\rhd A)\rhd B\wedge\Box C.
\end{align*}

As an illustration we shall calculate the first four principles.
\[
\begin{array}{lll}
\principle{R}^0 & :=  & A \rhd B \to \neg (A \rhd \neg C) \rhd B \wedge \Box C\\
\principle{R}^1 &:=& A \rhd B \to \Diamond \neg(D_1 \rhd \neg C) \wedge (D_1 \rhd A)  \rhd B \wedge \Box  C\\
\principle{R}^2 &:=& A \rhd B \to  \Diamond\Big[ (D_1 \rhd  D_2) \wedge\Diamond\neg(D_1 \rhd \neg C)\Big] \wedge (D_2 \rhd A) \rhd B \wedge \Box  C\\

\principle{R}^3 &:=& A \rhd B \to  \Diamond \Big( (D_2\rhd D_3) \wedge \Diamond\Big[ (D_1 \rhd  D_2) \wedge\Diamond\neg(D_1 \rhd \neg C)\Big] \Big) \wedge (D_3 \rhd A) \\
 & & \ \ \ \ \ \ \ \ \ \ \ \ \ \ \ \ \ \ \ \ \ \ \ \ \ \ \ \ \ \ \ \ \ \ \ \ \ \ \ \  \ \ \ \ \ \ \ \ \ \ \ \ \ \ \ \ \ \ \ \ \ \ \ \ \ \ \ \ \ \ \ \ \ \ \ \ \ \  \rhd B \wedge \Box  C

\end{array}
\]

It is not hard to determine the frame condition for the first couple of principles in this series and in Figure \ref{figure:broadSeries} we have depicted the first three frame-conditions. In this section we shall prove that the correspondence proceeds as expected. Informally, the frame condition for $\principle{R}^n$ shall be the universal closure of
\begin{equation}\label{equation:informalFrameConditionForBroadHierarchy}
x_{n+1}Rx_n \ldots Rx_0 Ry_0 S_{x_1} y_1 \ldots S_{x_{n}}y_{n} S_{x_{n+1}}y_{n+1} R z \Rightarrow y_0 S_{x_0}z.
\end{equation}

\begin{figure}[h]

\centering
\includegraphics[width=0.9\textwidth]{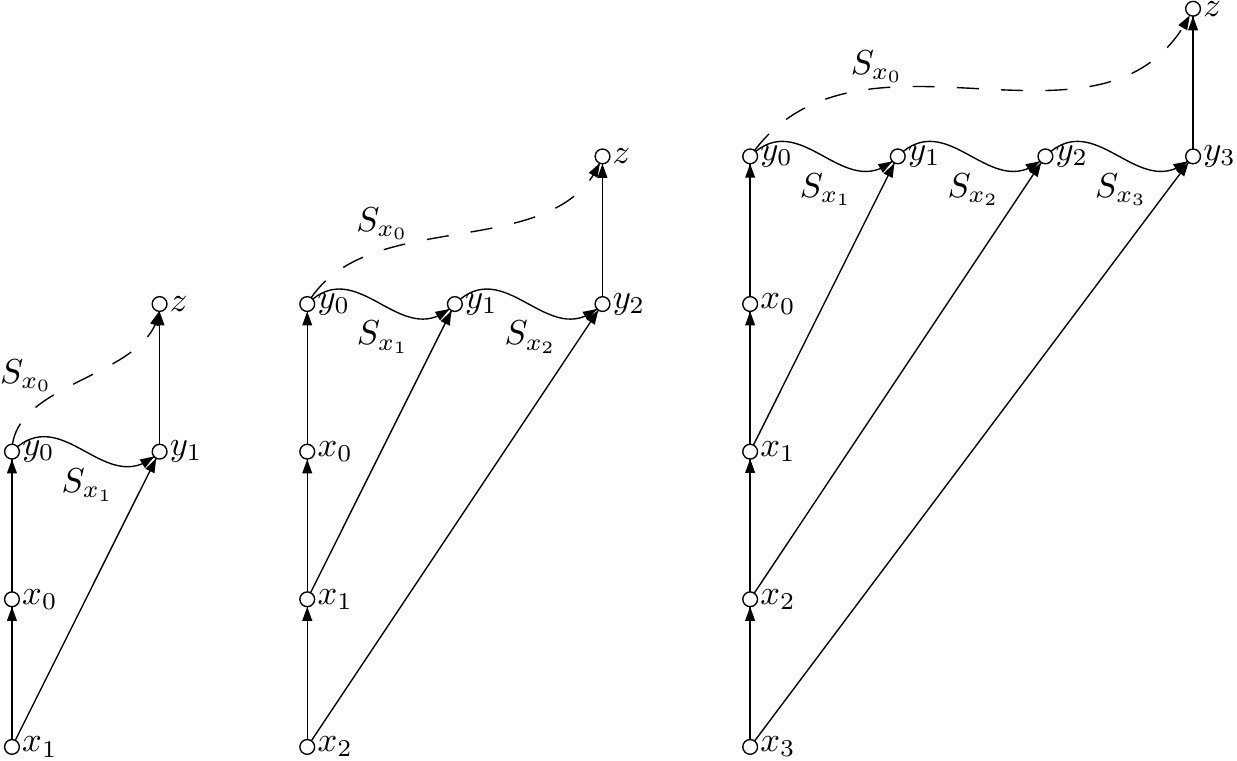}

\caption{From left to right, this figure depicts the frame conditions $(\principle{R}^0)$ through $(\principle{R}^2)$ corresponding to $\principle{R}^0$ through $\principle{R}^2$. The reading convention is as always: if all the un-dashed relations are present as in the picture, then also the dashed relation should be there.}\label{figure:broadSeries}
\end{figure}

We shall first recast the frame condition in a recursive fashion. In writing \eqref{equation:informalFrameConditionForBroadHierarchy} recursively we shall use those variables that will emphasise the relation with \eqref{equation:informalFrameConditionForBroadHierarchy}. Of course, free variables can be renamed at the readers liking.

First, we start by introducing a relation $\mathcal B_n$ that captures the antecedent of \eqref{equation:informalFrameConditionForBroadHierarchy}. Note that this antecedent says that first there is a chain of points $x_i$ related by $R$, followed by a chain of points $y_i$ related by different $S$ relations. The relation $\mathcal B_n$ will be applied to the end-points of both chains where the condition on the intermediate points is imposed by recursion.

\begin{align*}
\mathcal B_0(x_1,x_0,y_0,y_1) &:= x_1Rx_1Ry_0S_{x_1}y_1,\\
\mathcal B_{n+1}(x_{n+2},x_{0},y_0,y_{n+2}) &:= \exists x_{n+1}, y_{n+1} \big(x_{n+2}Rx_{n+1}
\ \& \  \mathcal B_n(x_{n+1},x_0,y_0,y_{n+1}) \\
 & \ \ \ \ \ \ \ \ \ \ \  \ \ \ \ \ \ \ \ \ \ \  \ \ \ \ \ \ \ \ \ \ \  \ \ \ \ \ \ \ \ \ \ \ \& \ y_{n+1}S_{x_{n+2}}y_{n+2} \big).
\end{align*}
For every $n\geq 0$ we can now define the first order frame condition $(\principle{R}^n)$ as follows.
\[
(\principle{R}^n) := \forall x_{n+1},x_0,y_0,y_{n+1} \ \big(\mathcal B_n(x_{n+1},x_0,y_0,y_{n+1})\Rightarrow \forall z\, (y_{n+1}Rz\Rightarrow y_0S_{x_0}z)\big)
.
\]

Sometimes we shall write $x_{n+1}\mathcal B_n[x_0,y_0] \, y_{n+1}$ conceiving the quaternary relation $\mathcal B_n$ as a binary relation indexed by the pair $x_0,y_0$. 
\begin{theorem}
We have $\mathfrak{F}\models (\principle{R}^n)$ if and only if $\mathfrak{F}\models \principle{R}^n,$
for each Veltman frame $\mathfrak{F}$ and $n\in\mathbb{N}.$
\end{theorem}

\subsection{Frame conditions for GVS}
In this section we present generalised-frame conditions for the above presented series $\principle{R}^i$. We observe that the mere definition of what it means to be a frame condition of an axiom scheme can be stated in second order logic where an arbitrary valuation corresponds to an arbitrary subset. Since GVS is second order in nature, it may raise a question what actually constitutes a natural frame condition for a principle other than just writing down the definition. This question is discussed in \cite{MasRovira:2020:MastersThesis}.
The \((\principle{R}^n)_{gen}\) condition reads as follows:
\begin{flalign*}
&\forall w,x_0,\ldots ,x_{n-1},y,z,\mathbb{A},\mathbb{B},\mathbb{C},\mathbb{D}_0,\ldots ,\mathbb{D}_{n-1}.\\
&wRx_{n-1}R\ldots Rx_0RyRz, \\
& (\forall u.wRu,u\in \mathbb{A}\Rightarrow \exists V.uS_wV\subseteq \mathbb{B}), \\
& (\forall u.x_{n-1}Ru\in \mathbb{D}_{n-1}\Rightarrow \exists V.uS_{x_{n-1}}V\subseteq \mathbb{A}), \\
& (\forall i\in \{1\ldots n-1\}\forall u.x_iRu\in \mathbb{D}_i\Rightarrow \exists V.uS_{x_i}V\subseteq \mathbb{D}_{i+1}), \\
& (\forall V.zS_yV\Rightarrow V\cap \mathbb{C}\neq 0),      \\
& z\in \mathbb{D}_0 \\
\Rightarrow \ & \exists V\subseteq \mathbb{B}.x_{n-1}S_wV,\{w:\exists v\in V.vRw\}\subseteq \mathbb{C}
\end{flalign*}
\begin{lemma}
\label{orgd348370}
Let $\mathfrak{M}$ be a generalised Veltman model, let \(x\) be a world and let \(n\in \mathbb{N}\). For any \(i\leq n\) we have
that if $\mathfrak{M},x \Vdash  U_i$ then there exist some worlds \(y,z,x_0,\ldots ,x_{i}\) such that:
\begin{enumerate}
\item $x_i=x$;
\item $x_iR\ldots Rx_0RyRz$;
\item for all $j\leq i$ we have that $\mathfrak{M},x_j\Vdash U_j$;
\item for all $j<i$ we have that $\mathfrak{M},x_j\Vdash D_j\rhd D_{j+1}$;
\item for all \(V\) we have that if $zS_yV$ then $V\cap \{w:\mathfrak{M},w\Vdash C\}\neq \emptyset $;
\item $\mathfrak{M},z\Vdash D_0$.
\end{enumerate}
\begin{proof}

By induction on \(i\).
For \(i=0\) we have that \(x\Vdash \Diamond \neg (D_0\rhd \neg C)\). It follows that there exists some
\(y\) such that \(xRy\Vdash \neg (D_0\rhd \neg C)\) and therefore there exists some \(z\) such that
\(yRz\Vdash D_0\) and for any \(V\), if \(zS_yV\), then \(V\cap \{w: \mathfrak{M},w\Vdash C\}\neq \emptyset \). It is clear
that all claims are met.

For \(i+1\) we have that \(x\Vdash \Diamond (D_i\rhd D_{i+1}\wedge U_i)\). It follows that there
exists some \(x_{i}\) such that \(x_i\Vdash D_i\rhd D_{i+1}\wedge U_i\). By the inductive hypothesis there exist
\(y,z,x_0,\ldots ,x_{i}\) such that satisfy claims \(1\ldots 6\). We set \(x_{i+1}\coloneqq x\). It is
trivial to observe that by using the inductive hypothesis all conditions are met for \(i+1\).
\end{proof}
\end{lemma}
\begin{theorem}
\label{orgede3f91}
For any generalised Veltman frame $\mathfrak{F},$ we have that $\mathfrak{F}$ satisfies
the \((\principle{R}^n)_{gen}\) condition if and only if any model based on $\mathfrak{F}$ forces every
instantiation of the \(\principle{R}^n\) principle. In symbols:
$$\mathfrak{F}\vDash (\principle{R}^n)_{gen}\ \ \ \mbox{if and only if} \ \ \ \mathfrak{F}\vDash \principle{R}^n$$
\end{theorem}
\begin{proof}
If \(n=0\) we refer to section \ref{sec:principles_models} or \cite{MasRovira:2020:MastersThesis} for full details.

For \(n+1\) proceed as follows.
Suppose first that we have a generalised Veltman model and a world \(w\) such that \(w\Vdash A\rhd B\).
Then assume also that \(wRx\Vdash ((D_n\rhd A)\wedge U_n)\). By Lemma \ref{orgd348370} it follows
that there exist \(y,z,x_0,\ldots ,x_{n}\) satisfying \(1\ldots 6\). For a formula $F$, define $\llbracket F\rrbracket \coloneqq \{ x\in \mathfrak{M} : x\Vdash F\}$. Then let \(\mathbb{A}\coloneqq \llbracket A\rrbracket \),
\(\mathbb{B}\coloneqq \llbracket B\rrbracket \), \(\mathbb{C}\coloneqq \llbracket C\rrbracket \) and for \(i\leq n\) let \(\mathbb{D}_i\coloneqq \llbracket D_i\rrbracket \).\\ 

It is routine to check that
the left part of the \((\principle{R}^{n+1})_{gen}\) holds and thus we get that there exists
some \(V\subseteq \mathbb{B}\) such that \(x_{n}S_wV\) and \(R[V]\subseteq \mathbb{C}\). Since \(V\subseteq \mathbb{B}\) we have that
\(x_{n}\Vdash B\) and since \(R[V]\subseteq \mathbb{C}\) we have \(x_{n}\Vdash \Box C\). Finally, since
\(x_{n}=x\) we conclude \(x\Vdash B\wedge \Box C\).

Let us now prove the opposite direction. Fix a generalised Veltman frame $\mathfrak{F}$ and let \(a,b,c,d_0,\ldots ,d_n\) be propositional
variables and assume $\mathfrak{F}\Vdash \principle{R}^{n+1}.$ Assume that the left part of the
implication of \((\principle{R}^{n+1})_{gen}\) holds. Now consider a model based on \(\mathfrak{F}\)
that satisfies the following:
$$\llbracket a\rrbracket=\mathbb{A}, \ \llbracket b\rrbracket= \mathbb{B}, \
   \llbracket c\rrbracket=\mathbb{C}, \ \llbracket d_i\rrbracket=\mathbb{D}_i, 
   \text{ for all } i\in \{0\ldots n\}$$

Now one can routinely check that \(w\Vdash A\rhd B\) and \(x\Vdash ((D_n\rhd A)\wedge U_n)\), hence there
exists \(U\) such that \(xS_wU\) and \(U\Vdash B\wedge \Box C\). From that we derive that \(U\subseteq \mathbb{B}\)
and \(R[U]\subseteq \mathbb{C}\).

This proof has been formalised in Agda. Full details can be found in \cite{MasRovira:2020:MastersThesis}.
\end{proof}
\section*{Acknowledgements}
We thank Dick de Jongh, Volodya Shavrukov, V\'it\v{e}slav \v{S}vejdar, and Rineke Verbrugge and Albert Visser for fruitful discussions and for comments on the 
origins of the field. Many thanks also to Taishi Kurahashi and Rineke Verbrugge for a careful cover-to-cover reading of the document pointing out various inaccuracies and room for improvement.

Joosten received support from grants RTC-2017-6740-7, FFI2015-70707P and 2017 SGR 270.
Mikec was supported by Croatian Science Foundation (HRZZ) under the projects UIP--05--2017--9219 
and IP--01--2018--7459.
Vukovi\'{c} was supported by Croatian Science Foundation (HRZZ) under the project 
IP--01--2018--7459.
\bibliographystyle{plain}
\bibliography{References}

\begin{thebibliography}{10}

\bibitem{ArecesHooglanddeJongh:2001}
C.~Areces, E.~Hoogland, and D.H.J. de~Jongh.
\newblock Interpolation, definability and fixed-points in interpretability
  logics.
\newblock In M.~Zakharyaschev, K.~Segerberg, M.~de~Rijke, and H.~Wansing,
  editors, {\em Advances in Modal Logic}, volume~2, pages 35--58. CSLI, 2001.

\bibitem{Bennet:1986Orderings}
C.~Bennet.
\newblock {\em On some orderings of extensions of arithmetic}.
\newblock Department of Philosophy, University of G{\"o}teborg, 1986.

\bibitem{Benthem1983}
J.~van Benthem.
\newblock {\em Modal Logic and Classical Logic}.
\newblock Bibliopolis, Naples, 1983.

\bibitem{van1984correspondence}
J.~van Benthem.
\newblock Correspondence theory.
\newblock In F.~Guenthner and D.~Gabbay, editors, {\em Handbook of
  Philosophical Logic}, pages 167--247. Springer, 1984.

\bibitem{Berarducci:1990:InterpretabilityLogicPA}
A.~Berarducci.
\newblock The {i}nterpretability {l}ogic of {p}eano {a}rithmetic.
\newblock {\em The Journal of Symbolic Logic}, 55(3):1059--1089, 1990.

\bibitem{BouJoosten:2011}
F.~Bou and J.J. Joosten.
\newblock The closed fragment of {IL} is {PSPACE} hard.
\newblock {\em Electronical Notes Theoretical Computer Science}, 278:47--54,
  2011.

\bibitem{Burgess02}
J.~P. Burgess.
\newblock Basic tense logic.
\newblock In D.~Gabbay and F.~Guenthner, editors, {\em Handbook of
  Philosophical Logic, Second Edition: Volume VII: Extensions of Classical
  Logic}, pages 1--42. Kluwer Academic Publishers, Dordrecht, 2002.

\bibitem{BilkvaGorisJoosten:2004:SmartLabels}
M.~Bílková, E.~Goris, and J.J. Joosten.
\newblock Smart labels.
\newblock In L.~Afanasiev and M.~Marx, editors, {\em Liber Amicorum for Dick de
  Jongh}. Intitute for Logic, Language and Computation, 2004.
\newblock Electronically published, ISBN: 90 5776 1289.

\bibitem{BilkovaJonghJoosten:2009:PRA}
M.~Bílková, D.H.J.~de Jongh, and J.J. Joosten.
\newblock Interpretability in \pra.
\newblock {\em Annals of Pure and Applied Logic}, 161(2):128--138, 2009.

\bibitem{CacicKovac:2014}
V.~{\v C}a\v{c}i\'{c} and V.~Kova\v{c}.
\newblock On the share of closed \il formulas which are also in \gl{}.
\newblock {\em Archive for Mathematical Logic}, 54:741--767, 2015.

\bibitem{CacicVrgoc:2013}
V.~{\v C}a\v{c}i\'{c} and D.~Vrgo\v{c}.
\newblock A {n}ote on {b}isimulation and {m}odal {e}quivalence in {p}rovability
  {l}ogic and {I}nterpretability {L}ogic.
\newblock {\em Studia Logica}, 101:31--44, 2011.

\bibitem{CacicVukovic:2012}
V.~{\v C}a\v{c}i\'{c} and M.~Vukovi\'{c}.
\newblock A note on normal forms for closed fragment of system \il.
\newblock {\em Mathematical Communications}, 17:195--204, 2012.

\bibitem{clot:arit93}
P.~Clote and J.~Kraj\'{\i}\v{c}ek, editors.
\newblock {\em Arithmetic, Proof Theory and Computational Complexity}. Oxford
  University Press, Oxford, 1993.

\bibitem{de1991proof}
D.H.J. de~Jongh, M.~Jumelet, and F.~Montagna.
\newblock On the proof of {S}olovay's theorem.
\newblock {\em Studia Logica}, 50(1):51--69, 1991.

\bibitem{dejo:solu98}
D.H.J. de~Jongh and D.~Pianigiani.
\newblock Solution of a problem of {D}avid {G}uaspari.
\newblock In E.~Or{\l}owska, editor, {\em {Logic at Work}}, Studies in
  Fuzziness and Soft Computing, pages 246--254. Physica-Verlag, Heidelberg/New
  York, 1998.

\bibitem{JonghVeltman83IntensionalLogicCourseText}
D.H.J. de~Jongh and F.J.M.M. Veltman.
\newblock Intensional logic.
\newblock University of Amsterdam, unpublished course text, 1983.

\bibitem{JonghVeltman:1990:ProvabilityLogicsForRelativeInterpretability}
D.H.J. de~Jongh and F.J.M.M. Veltman.
\newblock Provability logics for relative interpretability.
\newblock In P.P. Petkov, editor, {\em {M}athematical {L}ogic, {P}roceedings of
  the {H}eyting 1988 summer school in {V}arna, {B}ulgaria}, pages 31--42.
  Plenum Press, Boston, New York, 1990.

\bibitem{JonghVeltman:1999:ILW}
D.H.J. de~Jongh and F.J.M.M. Veltman.
\newblock Modal completeness of {IL${\sf W}$}.
\newblock In J.~Gerbrandy, M.~Marx, M.~Rijke, and Y.~Venema, editors, {\em
  Essays dedicated to Johan van Benthem on the occasion of his 50th birthday}.
  Amsterdam University Press, Amsterdam, 1999.

\bibitem{de2004completeness}
D.H.J. de~Jongh, F.J.M.M. Veltman, and R.~Verbrugge.
\newblock Completeness by construction for tense logics of linear time.
\newblock In J.~van Benthem, A.~Troelstra, F.J.M.M. Veltman, and A.~Visser,
  editors, {\em Liber Amicorum for Dick de Jongh}. Institute of Logic, Language
  and Computation, Amsterdam, 2004.

\bibitem{deJongh-Visser-91}
D.H.J. de~Jongh and A.~Visser.
\newblock Explicit fixed points in interpretability logic.
\newblock {\em Studia Logica: An International Journal for Symbolic Logic},
  50(1):39--49, 1991.

\bibitem{rijk:unar92}
M.~de~Rijke.
\newblock Unary interpretability logic.
\newblock {\em Notre Dame Journal of Formal Logic}, 33:249--272, 1992.

\bibitem{Japaridze:1992:ToleranceLogic}
G.~Dzhaparidze~(Japaridze).
\newblock The logic of linear tolerance.
\newblock {\em Studia Logica}, 51:249--277, 1992.

\bibitem{Japaridze:1993:WeakInterpretability}
G.~Dzhaparidze~(Japaridze).
\newblock A generalized notion of weak interpretability and the corresponding
  logic.
\newblock {\em Annals of Pure and Applied Logic}, 61:113--160, 1993.

\bibitem{DBLP:journals/ndjfl/Goris06}
E.~Goris.
\newblock Interpolation and the {i}nterpretability logic of {PA}.
\newblock {\em Notre Dame Journal of Formal Logic}, 47(2):179--195, 2006.

\bibitem{goris2020assuring}
E.~Goris, M.~Bílková, J.J. Joosten, and L.~Mikec.
\newblock Assuring and critical labels for relations between maximal consistent
  sets for interpretability logics, {\tt https://arxiv.org/2003.04623}, 2020.

\bibitem{Goris-Joosten-08}
E.~Goris and J.J. Joosten.
\newblock {Modal matters for interpretability logics}.
\newblock {\em Logic Journal of the IGPL}, 16(4):371--412, 08 2008.

\bibitem{GorisJoosten:2011:ANewPrinciple}
E.~Goris and J.J. Joosten.
\newblock A new principle in the interpretability logic of all reasonable
  arithmetical theories.
\newblock {\em Logic Journal of the IGPL}, 19(1):14--17, 2011.

\bibitem{GorisJoosten:2012:SelfProvers}
E.~Goris and J.J. Joosten.
\newblock Self provers and {$\Sigma_1$} sentences.
\newblock {\em Logic Journal of the IGPL}, 20(1):1--21, 2012.

\bibitem{GorisJoosten:2020:TwoSeries}
E.~Goris and J.J. Joosten.
\newblock Two new series of principles in the interpretability logic of all
  reasonable arithmetical theories.
\newblock {\em The Journal of Symbolic Logic}, 85(1):1--25, 2020.

\bibitem{haje:inte71}
P.~H{\'{a}}jek.
\newblock On interpretability in set theories {I}.
\newblock {\em Comm. Math. Univ. Carolinae}, 12:73--79, 1971.

\bibitem{haje:inte72}
P.~H{\'{a}}jek.
\newblock On interpretability in set theories {II}.
\newblock {\em Comm. Math. Univ. Carolinae}, 13:445--455, 1972.

\bibitem{haje:cons90}
P.~H{\'a}jek and F.~Montagna.
\newblock The logic of {$\Pi_1$}-conservativity.
\newblock {\em Archiv f{\"u}r Mathematische Logik und Grundlagenforschung},
  30:113--123, 1990.

\bibitem{haje:cons92}
P.~H{\'a}jek and F.~Montagna.
\newblock The logic of {$\Pi_1$}-conservativity continued.
\newblock {\em Archiv f{\"u}r Mathematische Logik und Grundlagenforschung},
  32:57--63, 1992.

\bibitem{HajekSvejdar:1991:ClosedFormulasInterpretability}
P.~H{\'a}jek and V.~\v{S}vejdar.
\newblock A note on the normal form of closed formulas of interpretability
  logic.
\newblock {\em Studia Logica}, 50(1):25--28, 1991.

\bibitem{HakoniemiJoosten:2016:TableauxForInterpretabilityLogics}
T.A. Hakoniemi and J.J. Joosten.
\newblock Labelled tableaux for interpretability logics.
\newblock In J.~van Eijck, R.~Iemhoff, and J.~J. Joosten, editors, {\em Liber
  Amicorum Alberti. A Tribute to Albert Visser}, pages 141--154. College
  Publications, London, 2016.

\bibitem{JoostenIcard:2012:RestrictedSubstitutions}
T.~Icard and J.J. Joosten.
\newblock Provabilty and interpretability logics with re- stricted
  substitutions.
\newblock {\em Notre Dame Journal of Formal Logic}, 53(2):133--154, 2012.

\bibitem{iemh:prop05}
R.~Iemhoff, D.H.J. de~Jongh, and C.~Zhou.
\newblock Properties of intuitionistic provability and preservativity logics.
\newblock {\em Logic Journal of IGPL}, 13(6):615--636, 2005.

\bibitem{igna:part91}
K.N. Ignatiev.
\newblock Partial conservativity and modal logics.
\newblock Technical Report X-91-04, ILLC, University of Amsterdam, 1991.

\bibitem{igna:prov93}
K.N. Ignatiev.
\newblock The provability logic of ${\Sigma}_1$-interpolability.
\newblock {\em Annals of Pure and Applied Logic}, 64:1--25, 1993.

\bibitem{JaparidzeJongh:1998:HandbookPaper}
G.~Japaridze and D.H.J. {de}~Jongh.
\newblock The logic of provability.
\newblock In S.~Buss, editor, {\em Handbook of proof theory}, pages 475--546.
  North-Holland Publishing Co., Amsterdam, 1998.

\bibitem{Japaridze:1994:SimpleProofPi1Conservativity}
G.~Japaridze~(Dzhaparidze).
\newblock A simple proof of arithmetical completeness for
  ${\Pi_1}$-conservativity logic.
\newblock {\em Notre Dame Journal of Formal Logic}, 35:346--354, 1994.

\bibitem{Joosten:1998:MasterThesis}
J.J. Joosten.
\newblock Towards the interpretability logic of all reasonable arithmetical
  theories.
\newblock Master's thesis, University of Amsterdam, 1998.

\bibitem{Joosten:2004:InterpretabilityFormalized}
J.J. Joosten.
\newblock {\em Interpretability Formalized}.
\newblock PhD thesis, Utrecht University, 2004.

\bibitem{Joosten:2005:ClosedFragmentILPRAwithIsig1}
J.J. Joosten.
\newblock The {c}losed {c}ragment of the {i}nterpretability {l}ogic of \pra
  with a constant for \isig{1}.
\newblock {\em Notre Dame Journal of Formal Logic}, 46(2):127--146, 2005.

\bibitem{Joosten:2010:ConsistencyInPRAandISIGMA}
J.J. Joosten.
\newblock Consistency statements and iterations of computable functions in
  $\isig{1}$ and \pra.
\newblock {\em Archive for Mathematical Logic}, 49(7,8):773--798, 2010.

\bibitem{Joosten:2016:OreyHajek}
J.J. Joosten.
\newblock On formalizations of the {O}rey-{H}\'ajek characterization for
  interpretability.
\newblock In P.~Cegielski, A.~Enayat, and R.~Kossak, editors, {\em Studies in
  Weak Arithmetics}, pages 57--90. CSLI Publications, Stanford, 2016.

\bibitem{JoostenMikecVisser:2020:TwoLogics}
J.J. Joosten, L.~Mikec, and A.~Visser.
\newblock Feferman axiomatisations, definable cuts and principles of
  interpretability.
\newblock {\em \mbox{forthcoming}}, 2020.

\bibitem{JoostenVisser:2000:IntLogicAll}
J.J. Joosten and A.~Visser.
\newblock The interpretability logic of all reasonable arithmetical theories.
  {T}he new conjecture.
\newblock {\em Erkenntnis}, 53(1-2):3--26, 2000.

\bibitem{JoostenVisser:2004:Toolkit}
J.J. Joosten and A.~Visser.
\newblock How to derive principles of interpretability logic, {A} toolkit.
\newblock In J.~van Benthem, F.J.M.M.~Veltman A.~Troelstra, and A.~Visser,
  editors, {\em Liber Amicorum for Dick de Jongh}. Intitute for Logic, Language
  and Computation, 2004.
\newblock Electronically published, ISBN: 90 5776 1289.

\bibitem{Kalsbeek:1991:TowardsExp}
M.B. Kalsbeek.
\newblock Towards the interpretability logic of
  $\mathrm{I}{\Delta}_0+\mathsf{{E}{X}{P}}$.
\newblock Logic Group Preprint Series~61, Faculty of Humanities, Philosophy,
  Utrecht University, {\tt https://lgps.sites.uu.nl}, 1991.

\bibitem{kurahashi2020modal}
T.~Kurahashi and Y.~Okawa.
\newblock Modal completeness of sublogics of the interpretability logic
  $\mathbf{IL}$, {\tt https://arxiv.org/2004.03813}, 2020.

\bibitem{lind:aspe97}
P.~Lindstr{\"o}m.
\newblock {\em {Aspects of Incompleteness}}, volume~10.
\newblock Springer, Berlin, 1997.

\bibitem{Litak14:trends}
T.~Litak.
\newblock Constructive modalities with provability smack.
\newblock In Guram Bezhanishvili, editor, {\em Leo Esakia on duality in modal
  and intuitionistic logics}, volume~4 of {\em Outstanding Contributions to
  Logic}, pages 179--208. Springer, 2014.

\bibitem{litak2017constructive}
T.~Litak.
\newblock Constructive modalities with provability smack (author's cut), 2017.
\newblock Unabridged and extended version of a chapter in the Esakia volume of
  "Outstanding Contributions to Logic".

\bibitem{LitakV18:im}
T.~Litak and A.~Visser.
\newblock Lewis meets {B}rouwer: constructive strict implication.
\newblock {\em Indagationes Mathematicae}, 29:36--90, February 2018.
\newblock A special issue "L.E.J. Brouwer, fifty years later".

\bibitem{LitakVisserInternal:2020:ForDick}
T.~Litak and A.~Visser.
\newblock {Lewisian Fixed Points I: Two Incomparable Constructions}.
\newblock This volume, 2020.

\bibitem{MasRovira:2020:MastersThesis}
J.~Mas~Rovira.
\newblock Frame {c}onditions for {i}nterpretability {l}ogics using
  {g}eneralised {v}eltman {s}emantics and the {a}gda {p}roof {a}ssistant.
\newblock Master's thesis, Master of Pure and Applied Logic, University of
  Barcelona, 2020.

\bibitem{luka-phd}
L.~Mikec.
\newblock {\em On logics and semantics for interpretability}.
\newblock PhD thesis, 2020 (expected).
\newblock University of Barcelona and University of Zagreb.

\bibitem{mikec-pakhomov-vukovic}
L.~Mikec, F.~Pakhomov, and M.~Vuković.
\newblock {Complexity of the interpretability logic \extil{}}.
\newblock {\em Logic Journal of the IGPL}, 27(1):1--7, 2018.

\bibitem{Mikec-Perkov-Vukovic-17}
L.~Mikec, T.~Perkov, and M.~Vuković.
\newblock Decidability of interpretability logics \extil{M_0} and \extil{W^*}.
\newblock {\em Logic Journal of the IGPL}, 25(5):758--772, 2017.

\bibitem{Mikec-Vukovic-20}
L.~Mikec and M.~Vuković.
\newblock Interpretability logics and generalised {V}eltman semantics.
\newblock {\em The Journal of Symbolic Logic}, to appear.

\bibitem{Montagna:1987:Provability}
F.~Montagna.
\newblock Provability in finite subtheories of {PA} and relative
  interpretability: a modal investigation.
\newblock {\em The Journal of Symbolic Logic}, 52(2):494--511, 1987.

\bibitem{Mycielski:latticeOfChapters:1990}
J.~Mycielski, P.~Pudl{\'a}k, and A.S. Stern.
\newblock {\em A lattice of chapters of mathematics (interpretations between
  theorems)}, volume 426 of {\em Memoirs of the American Mathematical Society}.
\newblock AMS, Providence, Rhode Island, 1990.

\bibitem{norell:thesis}
U.~Norell.
\newblock {\em Towards a practical programming language based on dependent type
  theory}.
\newblock PhD thesis, Department of Computer Science and Engineering, Chalmers
  University of Technology, SE-412 96 G\"{o}teborg, Sweden, September 2007.

\bibitem{orey:rela61}
S.~Orey.
\newblock Relative interpretations.
\newblock {\em Zeitschrift f{\"{u}}r mathematische Logik und Grundlagen der
  Mathematik}, 7:146--153, 1961.

\bibitem{PerkovVukovic:2014}
T.~Perkov and M.~Vuković.
\newblock A bisimulation characterization for interpretability logic.
\newblock {\em Logic Journal of the IGPL}, 22:872--879, 2014.

\bibitem{Perkov-Vukovic-16}
T.~Perkov and M.~Vuković.
\newblock Filtrations of generalized {Veltman} models.
\newblock {\em Mathematical Logic Quarterly}, 62(4-5):412--419, 2016.

\bibitem{Sasaki:2002:CutFreeIL}
K.~Sasaki.
\newblock A cut-free sequent system for the smallest interpretability logic.
\newblock {\em Studia Logica}, 70(3):353--372, 2002.

\bibitem{Shavrukov:1988:InterpretabilityLogicPA}
V.Y. Shavrukov.
\newblock The logic of relative interpretability over {P}eano arithmetic.
\newblock Preprint, Steklov Mathematical Institute, Moscow, 1988.
\newblock In Russian.

\bibitem{Shavrukov:1997:ReflexiveInfinitelyManyAxioms}
V.Y. Shavrukov.
\newblock Interpreting reflexive theories in finitely many axioms.
\newblock {\em Fundamenta Mathematicae}, 152:99--116, 1997.

\bibitem{DBLP:journals/jsyml/Strannegard99}
C.~Stranneg{\aa}rd.
\newblock Interpretability over {P}eano {A}rithmetic.
\newblock {\em The Journal of Symbolic Logic}, 64(4):1407--1425, 1999.

\bibitem{Svejdar:1978:DegreesOfInterpretability}
V.~{\v{S}}vejdar.
\newblock Degrees of interpretability.
\newblock {\em Commentationes Mathematicae Universitatis Carolinae},
  19:789--813, 1978.

\bibitem{svejdar:1983}
V.~{\v{S}}vejdar.
\newblock Modal analysis of generalized {R}osser sentences.
\newblock {\em The Journal of Symbolic Logic}, 48:986--999, 1983.

\bibitem{svej91}
V.~{\v S}vejdar.
\newblock Some independence results in interpretability logic.
\newblock {\em Studia Logica}, 50:29--38, 1991.

\bibitem{TarskiEtAl:1953:UndecidableTheories}
A.~Tarski, A.~Mostowski, and R.~Robinson.
\newblock {\em Undecidable Theories}.
\newblock North-Holland, Amsterdam, 1953.

\bibitem{Verbrugge}
L.C. Verbrugge.
\newblock Verzamelingen-{V}eltman frames en modellen (set {V}eltman frames and
  models).
\newblock Unpublished manuscript, Amsterdam, 1992.

\bibitem{Verbrugge1988MastersThesis}
L.C. Verbrugge.
\newblock Does {S}olovay’s completeness theorem extend to bounded arithmetic?
\newblock Master's thesis, University of Amsterdam, 1988.

\bibitem{verb:feas93}
R.~Verbrugge.
\newblock Feasible interpretability.
\newblock In {\em \cite{clot:arit93}}, pages 387--428. 1993.

\bibitem{Visser:1988:preliminaryNotesOnInterpretabilityLogic}
A.~Visser.
\newblock Preliminary notes on {I}nterpretability {L}ogic.
\newblock Technical Report LGPS 29, Department of Philosophy, Utrecht
  University, 1988.

\bibitem{Visser:1990:InterpretabilityLogic}
A.~Visser.
\newblock {I}nterpretability logic.
\newblock In P.P. Petkov, editor, {\em {M}athematical {L}ogic, {P}roceedings of
  the {H}eyting 1988 summer school in {V}arna, {B}ulgaria}, pages 175--209.
  Plenum Press, Boston, New York, 1990.

\bibitem{Visser:1991:FormalizationOfInterpretability}
A.~Visser.
\newblock The formalization of interpretability.
\newblock {\em Studia Logica}, 50(1):81--106, 1991.

\bibitem{Visser:1997:OverviewIL}
A.~Visser.
\newblock An overview of {I}nterpretability {L}ogic.
\newblock In M.~Kracht, M.~{de} Rijke, and H.~Wansing, editors, {\em Advances
  in modal logic '96}, pages 307--359. CSLI Publications, Stanford, CA, 1997.

\bibitem{viss:faith05}
A.~Visser.
\newblock Faith {\&} {F}alsity: a study of {F}aithful {I}nterpretations and
  false ${\Sigma}^0_1$-sentences.
\newblock {\em Annals of Pure and Applied Logic}, 131(1--3):103--131, 2005.

\bibitem{viss:cate06}
A.~Visser.
\newblock Categories of {t}heories and {i}nterpretations.
\newblock In A.~Enayat, I.~Kalantari, and M.~Moniri, editors, {\em {Logic in
  {T}ehran. {P}roceedings of the workshop and conference on {L}ogic, {A}lgebra
  and {A}rithmetic, held {O}ctober 18--22, 2003}}, volume~26 of {\em Lecture
  {N}otes in {L}ogic}, pages 284--341. ASL, A.K. Peters, Ltd., Wellesley,
  Mass., 2006.

\bibitem{Visser:2014:IntyDegreesFiniteSequential}
A.~Visser.
\newblock Interpretability degrees of finitely axiomatized sequential theories.
\newblock {\em Archive for Mathematical Logic}, 53:23--42, 2014.

\bibitem{Vrgoc-Vukovic}
D.~Vrgoč and M.~Vuković.
\newblock {Bisimulations and bisimulation quotients of generalized Veltman
  models}.
\newblock {\em Logic Journal of the IGPL}, 18(6):870--880, 2009.

\bibitem{Vukovic96}
M.~Vuković.
\newblock Some correspondence of principles in interpretability logic.
\newblock {\em Glasnik Matemati\v{c}ki}, 31(51):193--200, 1996.

\bibitem{Vukovic99}
M.~Vuković.
\newblock The principles of interpretability.
\newblock {\em Notre Dame Journal of Formal Logic}, 40(2):227--235, 1999.

\bibitem{Vukovic08}
M.~Vuković.
\newblock Bisimulations between generalized {V}eltman models and {V}eltman
  models.
\newblock {\em Mathematical Logic Quarterly}, 54(4):368--373, 2008.

\bibitem{wadler2015propositions}
P.~Wadler.
\newblock Propositions as types.
\newblock {\em Communications of the ACM}, 58(12):75--84, 2015.

\bibitem{Zambella:1992:Interpretability}
D.~Zambella.
\newblock On the proofs of arithmetical completeness of interpretability logic.
\newblock {\em Notre Dame Journal of Formal Logic}, 35:542--551, 1992.

\end{thebibliography}
\end{document}